\documentclass[12pt]{iopart}

\usepackage{algorithm}
\usepackage{algorithmic} 
\usepackage{tikz}
\usepackage{iopams}
\usepackage[framed,numbered,autolinebreaks,useliterate]{mcode}

\expandafter\let\csname equation*\endcsname\relax

\expandafter\let\csname endequation*\endcsname\relax

\usepackage{amsmath}
\usepackage{amsopn}
\usepackage{amsthm}
\usepackage{todonotes}
\usepackage{mathtools}
\usepackage{graphicx}
\usepackage{array}
\usepackage{hyperref}
\usepackage[noabbrev]{cleveref} 
\usepackage{color}
\usepackage{subcaption}
\usepackage[overload]{empheq}

\allowdisplaybreaks 
\crefformat{equation}{(#2#1#3)}
\crefformat{appendix}{#2#1#3}

\newtheorem{theorem}{Theorem}[section]
\newtheorem{lemma}[theorem]{Lemma}

\newtheorem{remark}[theorem]{Remark}
\DeclareMathOperator*{\argmin}{argmin}

\begin{document}

\title[Projected Newton Method]{Sequential Projected Newton method for regularization of nonlinear least squares problems}

\author{J Cornelis \& W Vanroose}

\address{Department of Mathematics, University of Antwerp, 2020 Antwerp, Belgium}
\ead{jeffrey.cornelis@uantwerp.be}
\vspace{10pt}
\begin{indented}
\item[]\today
\end{indented}

\begin{abstract}
We develop a computationally efficient algorithm for the automatic regularization of nonlinear inverse problems based on the discrepancy principle. We formulate the problem as an equality constrained optimization problem, where the constraint is given by a least squares data fidelity term and expresses the discrepancy principle. The objective function is a convex regularization function that incorporates some prior knowledge, such as the total variation regularization function. Using the Jacobian matrix of the nonlinear forward model, we consider a sequence of quadratically constrained optimization problems that can all be solved using the Projected Newton method. We show that the solution of such a quadratically constrained sub-problem results in a descent direction for an exact merit function. This merit function can then be used to describe a formal line-search method. We also formulate a slightly more heuristic approach that simplifies the algorithm and allows for an inexact solution of the sequence of sub-problems. We illustrate the behavior of the algorithm using a number of numerical experiments, with Talbot-Lau X-ray phase contrast imaging as the main application. The numerical experiments confirm that the quadratically constrained sub-problems need not be solved with high accuracy in early iterations to make sufficient progress towards the solution.
In addition, we show that the proposed method is able to produce reconstructions of similar quality compared to other state-of-the-art approaches with a significant reduction in computational time.
\end{abstract}

\vspace{2pc}

\section{Introduction}

In this manuscript we consider the regularized nonlinear least squares problem
\begin{equation} \label{eq:ireg1}
\min_{x\in\mathbb{R}^n} \Psi(x) \hspace{0.5cm}\text{subject to}\hspace{0.5cm}\frac{1}{2}||s(x) - b||^2 \leq \frac{\sigma^2}{2}
\end{equation}
where $\Psi: \mathbb{R}^n \rightarrow \mathbb{R}$ is a convex twice continuously differentiable function. The function $s :\mathbb{R}^n \rightarrow \mathbb{R}^m$ with $m\geq n$ models some nonlinear forward operation and $b\in\mathbb{R}^m$ is some data contaminated with measurement errors or noise. Here, $\sigma$ is an estimate of the noise-level, i.e. $\sigma \approx ||b_{ex} - b||$, for some ``exact'' data $b_{ex}$. In the following we assume that the unconstrained solution of the least squares problem is below the noise-level, i.e. that $\min_{x\in\mathbb{R}^n} \|s(x) - b\| \leq \sigma$. This implies that the constraint in \cref{eq:ireg1} is feasible and thus there exists at least one local solution. 

Furthermore, we assume $s(x)$ has component functions $s_i(x):\mathbb{R}^n \rightarrow \mathbb{R}$ for $i=1,\ldots,m$ that are all continuously differentiable, such that the Jacobian matrix 
\begin{equation*}
J(x) = \left( \frac{\partial s_i(x)}{\partial x_j} \right)_{i,j} = \begin{pmatrix} \frac{\partial s_1(x)}{\partial x_1} &  \ldots & \frac{\partial s_1(x)}{\partial x_n} \\
\vdots & \ddots & \vdots \\
\frac{\partial s_m(x)}{\partial x_1} &  \ldots & \frac{\partial s_m(x)}{\partial x_n}
\end{pmatrix} \in \mathbb{R}^{m \times n}
\end{equation*}
is well-defined. Let us denote the constraint function as $c(x) = \frac{1}{2}||s(x) - b||^2 - \frac{\sigma^2}{2}$. It is an easy calculation to verify that the gradient of the constraint function is given by $\nabla c(x) = J(x)^T(s(x) - b)$. Note that the constraint function $c(x)$ is not necessarily convex. 

In most cases of interest we can show that the inequality constraint in \cref{eq:ireg1} becomes active for the solution $x^{*}$ and can thus be replaced by an equality constraint. 

\begin{lemma}\label{thm:unique_lag}

Suppose $x^*$ is a local solution of \cref{eq:ireg1} with $\nabla \Psi(x^*) \neq 0$ and $\nabla c(x^*) \neq 0$. Then $||s(x^*) - b|| = \sigma$ and there exists a unique Lagrange multiplier $\lambda^*>0$ such that 

\begin{equation}\label{eq:kkt1}
\left\{ \begin{array}{l}\nabla \Psi(x^{*}) + \lambda^{*}\nabla c(x^{*}) = 0,   \\
c(x^{*}) = 0. \end{array} \right. 
\end{equation}

\end{lemma}

\begin{proof}
The linear independence constraint qualification (LICQ) \cite{nocedal2006numerical} trivially holds for $x^{*}$ since we only have one constraint and we assume that $\nabla c(x^*) \neq 0$ \footnote{LICQ is an assumption that ensures that the Karush-Kuhn-Tucker conditions are necessary conditions for local optimality of a constrained optimization problem, see definition 12.4 and theorem 12.1 in \cite{nocedal2006numerical}.}. Hence, there exists a Lagrange multiplier $\lambda^{*}\in\mathbb{R}$ such that the Karuhn-Kush-Tucker or KKT conditions are satisfied :

\begin{equation*}
\left\{ \begin{array}{l}\nabla \Psi(x^*) + \lambda^* \nabla c(x^*) = 0, \\
\lambda^* c(x^*) = 0 \\
c(x^*)\leq 0,\hspace{0.2cm} \lambda^*\geq 0.  \end{array}\right. 
\end{equation*}

Suppose that $||s(x^*) - b||\neq \sigma$, i.e. $c(x^{*})\neq0$, then it follows from the complementarity condition that $\lambda^* = 0$. This is turn means that $\nabla \Psi(x^{*})= 0$ which contradicts our assumption that $\nabla \Psi(x^*) \neq 0$. Hence, it follows that $\lambda^* > 0$ and $||s(x^*) - b|| = \sigma$ and that $(x^{*},\lambda^*)$ satisfies \eqref{eq:kkt1}. Uniqueness of the Lagrange multiplier easily follows from the fact $\nabla c(x^*) \neq 0$. 
\end{proof}
The conditions in \cref{thm:unique_lag} are very mild since $\nabla \Psi(x^{*})= 0$ means that $x^*$ is an unconstrained minimizer of the convex regularization function.  Similarly we have that $\nabla c(x^*)=0$ expresses the first order necessary optimality condition of the unconstrained optimization problem $\argmin ||s(x) - b||$.

Note that the equations in \eqref{eq:kkt1} are precisely the KKT conditions of the equality constrained problem
\begin{equation} \label{eq:reg1}
\min_{x\in\mathbb{R}^n} \Psi(x) \hspace{0.5cm}\text{subject to}\hspace{0.5cm}\frac{1}{2}||s(x) - b||^2 = \frac{\sigma^2}{2}.
\end{equation}
Here, the constraint expresses the discrepancy principle \cite{hansen2010discrete}, which basically says that a point $x$ that satisfies $||s(x) - b||<\sigma$ is expected to be `over-fitted' to the noisy data $b$. Among all possible solutions $x$ that fit data $b$ `well enough' without being over-fitted, the one that minimizes the regularization function $\Psi(x)$ is chosen. 

The equality constrained optimization problem \cref{eq:reg1} is closely related to the (unconstrained) regularized nonlinear least squares problem
\begin{equation} \label{eq:reg_alpha}
\min_{x\in\mathbb{R}^n} \frac{1}{2} \| s(x) - b\|^2 + \alpha\Psi(x),
\end{equation}
with fixed regularization parameter $\alpha>0$. The first order optimality conditions of \cref{eq:reg_alpha} are given by
\begin{equation*}
\underbrace{J(x^*)^T(s(x^*) - b)}_{=\nabla c(x^*)} + \alpha \nabla \Psi(x^*) = 0. 
\end{equation*}
Any point $x^*$ that satisfies these equations also satisfies the first equation of the KKT conditions \cref{eq:kkt1} of the equality constrained optimization problem \cref{eq:reg1} with $\lambda^* = 1/\alpha$. Hence, solving \cref{eq:reg1} can be seen as an approach to simultaneously solve the regularized nonlinear least squares problem \cref{eq:reg_alpha} and find the corresponding regularization parameter $\alpha$ such that the discrepancy principle is satisfied. Other well-known parameter choice methods include generalized cross validation and the L-curve criterion. We refer to \cite{hansen2010discrete} for more general information on the different parameter choice methods. 

In this work we consider a novel approach to compute a solution $(x^{*},\lambda^{*})$ to the nonlinear system of equations \cref{eq:kkt1} by solving a sequence of sub-problems, by using the linear approximation $s(x + p)\approx s(x) + J(x)p$  for the constraint in \cref{eq:reg1} for some $x\in \mathbb{R}^n$. By Taylor's theorem we have
\begin{equation*}
s(x + p) = s(x) + \int_0^1 J(x + tp)p dt. 
\end{equation*}
Hence it follows that 
\begin{align}
||s(x) + J(x)p - s(x + p)|| &= || J(x)p - \int_0^1 J(x + tp)p dt || \notag \\
&= ||\int_0^1 \left(J(x) - J(x + tp)\right)  p dt || \notag \\ 
&\leq ||p|| \max_{t\in[0,1]} || J(x) - J(x + tp)|| =  o(||p||). \label{eq:littleo}
\end{align}
When $J(x)$ is Lipschitz continuous the bound is replaced by $\mathcal{O}(||p||^2)$. 

After some calculations it follows that we have the following equivalent expressions: 
\begin{equation*}
\frac{1}{2}||J(x)p  + s(x) - b||^2 = \frac{\sigma^2}{2} \Leftrightarrow c(x) + \nabla c(x)^T p + \frac{1}{2}p^T J(x)^T J(x)p = 0.
\end{equation*}
Hence it is clear that taking a linear approximation of the forward model $s(x + p)$ corresponds to taking a quadratic model of the constraint function $c(x + p)$. Alternatively we could have immediately considered a quadratic model of $c(x + p)$ using Taylor's theorem:
\begin{equation*}
c(x + p)\approx c(x) + \nabla c(x)^T p +  \frac{1}{2}p^T \nabla^2 c(x) p. 
\end{equation*}
However, since $c(x)$ is in general non-convex we have that the Hessian is possibly indefinite, which is undesirable for optimization algorithms.  Moreover, for a least squares function we know that the matrix $J(x)^TJ(x)$ is an approximation of the true Hessian $\nabla^2 c(x)$, see for instance chapter 10 in \cite{nocedal2006numerical}. This approximation is for instance used in the Gauss-Newton algorithm for solving a least squares minimization problem. 

The main idea of the method proposed in this manuscript is to solve a sequence of problems of the form
\begin{equation} \label{eq:reg2}
\min_{p\in\mathbb{R}^n} \Psi(x + p) \hspace{0.5cm}\text{subject to}\hspace{0.5cm}\frac{1}{2}||J(x)p  + s(x) - b||^2 = \frac{\sigma^2}{2}
\end{equation}
for a given iterate $x\in\mathbb{R}^n$.
The KKT conditions for this optimization problem are given by  
\begin{subequations} \label{eq:kkt2}
\begin{align}[left = \empheqlbrace\,]
& \nabla \Psi(x + p) + \lambda (\nabla c(x) + J(x)^T J(x) p) = 0, \label{eq:kkt2a}\\
& c(x) + \nabla c(x)^T p + \frac{1}{2}p^T J(x)^T J(x)p = 0, \label{eq:kkt2b}
\end{align}
\end{subequations}
where $\lambda\in\mathbb{R}$ is the Lagrange multiplier.

Alternatively, we could also consider a quadratic model of the objective function and solve a sequence of problems
\begin{equation}\label{eq:quadmodel}
\min_{p\in\mathbb{R}^n} \nabla \Psi(x)^T p + \frac{1}{2} p^T \nabla^2 \Psi(x) p \hspace{0.5cm}\text{subject to}\hspace{0.5cm}\frac{1}{2}||J(x)p  + s(x) - b||^2 = \frac{\sigma^2}{2}.
\end{equation}
This idea forms the basis of one of the reference methods that we consider in this manuscript. In this case we get an approach that closely resembles the one taken in \cite{fukushima2003sequential}, namely the Sequential Quadratically Constained Quadratic programming (SQCQP) method. The SQCQP method is slightly more general in the sense that they work with possibly multiple constraint functions $c_{i}(x):\mathbb{R}^n\rightarrow\mathbb{R}$ for $i=1,\ldots,\tilde{m}$. However, they assume convexity of these functions and then use the true Hessians $\nabla^2 c_{i}(x)$ in the quadratic model. With the numerical experiments, see \cref{sec:num}, it will become clear that in the context of the current manuscript it is better to solve a sequence of problems of the form \cref{eq:reg2}. Hence, we focus our presentation on this case. However, many of the results presented here can be modified in case we solve a sequence of problems \cref{eq:quadmodel}. 

A second reference method that we consider for the numerical experiments is based on the approximation $J(x)^TJ(x)\approx \nabla^2 c(x)$. The algorithm can be seen as a type of Gauss-Newton method and is able to solve the regularized nonlinear inverse problem with fixed regularization parameter $\alpha>0$ shown in \cref{eq:reg_alpha}. The linear systems of equations that appear in each iteration can be solved using the Conjugate Gradient method \cite{hestenes1952methods}.  

Note that even when \eqref{eq:kkt1} has a solution, this does not necessarily imply that \eqref{eq:kkt2} has a solution for all $x$. However, using the implicit function theorem, we can show that this sub-problem has a solution for all $x$ in a neighborhood of the solution $x^{*}$.
The implicit function theorem can be formulated as follows: 
\begin{theorem} \label{thm:IFT}
Let $\Phi:\mathbb{R}^{n}\times\mathbb{R}^{q}\rightarrow \mathbb{R}^q$ be a continuously differentiable function and let $\mathbb{R}^{n}\times\mathbb{R}^{q}$ have coordinates $(x,y)$. Suppose $(x^{*},y^{*})$ is a point satisfying $\Phi(x^{*},y^{*})=0$. If the Jacobian matrix of $\Phi$ with respect to $y$ is nonsingular in $(x^{*},y^{*})$, i.e. if the matrix $J_{\Phi}(x^{*},y^*) := \left(\frac{\partial \Phi_i (x^{*},y^{*})}{\partial y_j} \right)_{i,j} \in\mathbb{R}^{q \times q}$ is invertible, then there exists some open set $X\subset\mathbb{R}^n$ containing $x^{*}$ such that there exists a unique continuously differentiable function $g:X\rightarrow \mathbb{R}^{q}$ such that $g(x^{*}) = y^{*}$ and $\Phi(x,g(x)) = 0$ for all $x\in X$.  
\end{theorem}
\begin{proof} See for instance \cite{dontchev2009implicit,fitzpatrick2009advanced}.  
\end{proof}

\begin{lemma} \label{thm:exist}
Let $(x^*,\lambda^*)$ be a KKT point \cref{eq:kkt1}  with $\nabla c(x^*) \neq 0$ and suppose the matrix $\nabla^2 \Psi(x^*) + \lambda^*J(x^*)^TJ(x^{*})$ is positive definite. Then there exist an open open ball $\mathcal{B}_\rho(x^{*})$ with center $x^*$ and radius $\rho>0$ such that for all $x\in \mathcal{B}_\rho(x^{*})$ there exists a solution $(p,\lambda)$ for \cref{eq:kkt2} with $\lambda>0$. 
\end{lemma}
\begin{proof}
To prove this claim we apply the implicit function theorem to the function $\Phi: \mathbb{R}^{n}\times\mathbb{R}^{n+1} \rightarrow  \mathbb{R}^{n + 1}$ with
\begin{equation*}
\Phi(x,p,\lambda) = \begin{pmatrix} \nabla \Psi(x + p) + \lambda \nabla c(x) + \lambda J(x)^TJ(x)p \\ c(x) + \nabla c(x)^T p + \frac{1}{2}p^T J(x)^TJ(x)p\end{pmatrix}.
\end{equation*}
Obviously we have $\Phi(x^*,0,\lambda^*)= 0$ since $(x^{*},\lambda^*)$ solves \cref{eq:kkt1}. Moreover we have that the Jacobian of $\Phi(x,p,\lambda)$ with respect to $(p,\lambda)$ is given by
\begin{equation*}
J_{\Phi}(x,p,\lambda) =  \begin{pmatrix}
\nabla^2 \Psi(x + p) + \lambda J(x)^T J(x) & J(x)^TJ(x)p + \nabla c(x) \\  p^T J(x)^TJ(x) + \nabla c(x)^T & 0
\end{pmatrix}.
\end{equation*}
Evaluating this matrix in $(x^*,0,\lambda^*)$ gives us 
\begin{equation*}
J_{\Phi}(x^*,0,\lambda^*) =  \begin{pmatrix}
\nabla^2 \Psi(x^*) + \lambda^* J(x^*)^T J(x^*) & \nabla c(x^*) \\  \nabla c(x^*)^T & 0
\end{pmatrix}
\end{equation*}
which is non-singular since $\nabla^2 \Psi(x^*) + \lambda^*J(x^*)^TJ(x^*)$ is positive definite and $\nabla c(x^*)\neq 0$. Hence, from \cref{thm:IFT} it follows that there exists an open set $X\subset\mathbb{R}^n$ containing $x^{*}$ such that there exists a unique continuously differentiable function $g:X\rightarrow \mathbb{R}^{n + 1}$ such that $g(x^{*}) = (0,\lambda^*)$ and $\Phi(x,g(x)) = 0$ for all $x\in X$.  It is now clear that $g(x) = (p,\lambda)$ is a solution for \cref{eq:kkt2}. Furthermore, by continuity of $g(x)$ and the fact that $\lambda^*>0$ we can choose the any radius $\rho>0$ small enough such that $\mathcal{B}_\rho(x^{*}) \subset X$ and $\lambda>0$.
\end{proof}

The rest of the paper is organized as follows. In \cref{sec:deriv} we derive the Sequential Projected Newton method, which is based on (approximately) solving a sequence of problems of the form \cref{eq:reg2}. In \cref{sec:reference_methods} we give some comments on related work and describe two reference methods. Next, in \cref{sec:talbot}, we introduce the application that we consider for our numerical experiments, namely Talbot-Lau X-ray phase constrast imaging. In \cref{sec:num} we perform a number of experiments with the proposed algorithm and compare it with the two reference methods described in \cref{sec:reference_methods}. First, we show that the sub-problems \cref{eq:reg2} do not need to be solved with a high accuracy to make reasonable progress towards the solution. Next, we illustrate the fact that the Sequential Projected Newton method is significantly less computationally expensive compared to the two reference methods, while producing solutions of similar quality. Lastly this work is concluded and an outlook is given in \cref{sec:conclusion}.

\section{Derivation of the Sequential Projected Newton method} \label{sec:deriv}
In this section we describe a line-search strategy that can be used when solving a sequence of problems of the form \cref{eq:kkt2}. Let us first describe a good choice of initial point $x_0$ to start our algorithm. We know that there exists a point $x_0$ that satisfies $||s(x_0) - b||\leq \sigma$, since we assumed that the unconstrained solution is below the noise-level. Moreover, such a point can be computed very easily by performing just a few Gauss-Newton iterations applied to $\argmin_{x\in\mathbb{R}^n}||s(x) - b||$. For such a point we know that the linearized inequality constraint $||J(x_0)p+ s(x_0) - b||\leq \sigma$ is feasible, since $p=0$ satisfies the inequality. From this it follows that the equality constraint $||J(x_0)p + s(x_0) - b||=\sigma$ is also feasible, since the constraint can be seen as a convex paraboloid. This means that the nonlinear system of equations \cref{eq:kkt2} has a solution for $x = x_0$ and hence we can use this as our initial point. See \cref{thm:initialpoint} for more details on how we can compute such a point. 

Let us denote $ [ \cdot ]_{+} = \max\{0,\cdot\}$. It is well known that the following function is an exact merit function for the optimization problem \cref{eq:ireg1}: 
\begin{equation*}
F_r(x) = \Psi(x) + r [c(x)]_+ . 
\end{equation*}
The authors in \cite{fukushima2003sequential} use this merit function in proving global convergence of the SQCQP method. The analysis in this section in inspired by this work. 
 
In this section we show that if we choose the penalty parameter $r>0$ large enough, we can then always find a step-length $\beta\in (0,1]$ such that $F_r(x + \beta p) < F_r(x)$. 
We start by proving the following lemma: 
\begin{lemma}\label{thm:lem}
For all $x\in\mathcal{B}_\rho(x^*)$ there exists a constant $\phi>0$ such that
\begin{equation*}
\Psi(x + \beta p) - \Psi(x) \leq - \frac{\beta\lambda}{2} p^T J(x)^TJ(x)p + \lambda \beta c(x) + \phi \beta^2 ||p||^2
\end{equation*}
where $(p,\lambda)$ is the KKT pair \cref{eq:kkt2}. 
\end{lemma}
\begin{proof}
By Taylor's theorem we have that there exists some $t_\beta\in(0,1)$ such that 
\begin{equation*}
\Psi(x + \beta p) = \Psi(x) + \beta \nabla\Psi(x)^T p + \frac{\beta^2}{2}p^T \nabla^2 \Psi(x + t_\beta \beta p) p
\end{equation*}
and thus we have for all $\beta$ that 
\begin{equation}\label{eq:lem1}
\Psi(x + \beta p) - \Psi(x) \leq \beta \nabla\Psi(x)^T p + \phi \beta^2 ||p||^2
\end{equation} 
with $\phi = \max_{t\in(0,1)} \frac{1}{2}\|\nabla^2 \Psi(x + t p)\|$.
Moreover, we also have
\begin{equation*}
\nabla \Psi(x + p) = \nabla \Psi(x) + \int_{0}^1 \nabla^2 \Psi(x + tp)pdt
\end{equation*} 
which implies $\nabla \Psi(x)^Tp \leq \nabla \Psi(x + p)^T p$ since $\nabla^2 \Psi(x)$ is positive semi-definite. 
Using this inequality together with \cref{eq:lem1} and the KKT conditions \cref{eq:kkt2}, we get
\begin{align*}
\Psi(x + \beta p) - \Psi(x)  &\leq  \beta \nabla\Psi(x + p)^T p + \phi \beta^2 ||p||^2 \\
&= \lambda\beta \left(- \nabla c(x)^T p - p^TJ(x)^T J(x)p\right) + \phi \beta^2 ||p||^2 \\
&= \lambda \beta c(x) -\frac{\beta\lambda}{2} p^T J(x)^TJ(x)p + \phi \beta^2 ||p||^2.
\end{align*}
\end{proof}

\begin{lemma}\label{thm:descent}
Let $x\in\mathcal{B}_\rho(x^*)$ with corresponding KKT pair $(p,\lambda)$ that satisfies \cref{eq:kkt2} and $r>\lambda$. Let $0<\eta<1$, then we have for all $\beta$ sufficiently small
\begin{equation} \label{eq:descent}
F_r(x + \beta p) - F_r(x) \leq - \frac{\eta\lambda\beta}{2}  p^T  J(x)^T J(x) p.
\end{equation}
    
\end{lemma}
\begin{proof}
From \cref{thm:lem} we have that there exists a constant $\phi_1>0$ such that
\begin{equation*}
\Psi(x + \beta p) - \Psi(x) \leq - \frac{\beta\lambda}{2} p^T J(x)^TJ(x)p + \lambda \beta c(x) + \phi_1 \beta^2 ||p||^2.
\end{equation*}
From the second equation in \cref{eq:kkt2} it follows that $c(x) + \nabla c(x)^T p \leq 0$. Hence, for all $\beta\in[0,1]$ we have
\begin{equation*}
c(x) + \beta\nabla c(x)^T p \leq (1-\beta) c(x). 
\end{equation*}
From this it also follows that 
\begin{equation*}
[c(x) + \beta\nabla c(x)^T p]_+ -  [c(x)]_+ \leq - \beta [c(x)]_+ . 
\end{equation*}
Again using Taylor's theorem we have for some $t_\beta\in(0,1)$ that
\begin{align*}
[c(x + \beta p ) ]_+ &= \left[c(x) + \beta \nabla c(x)^T p + \frac{\beta^2}{2}p^T \nabla c^2 (x + t_\beta\beta p)p  \right]_+ \\ 
&\leq \left[c(x) + \beta \nabla c(x)^T p\right]_+ +  \left[\frac{\beta^2}{2}p^T \nabla c^2 (x + t_\beta\beta p) p \right]_+ \\
& \leq \left[c(x) + \beta \nabla c(x)^T p\right]_+ +\phi_2 \beta^2 ||p||^2
\end{align*}
with $\phi_2 = \max_{t\in(0,1)} \frac{1}{2}\|\nabla^2 c(x + t p)\|$.
As a consequence we also have
\begin{align*}
[c(x+\beta p)]_+ -  [c(x)]_+ &\leq  \left[c(x) + \beta \nabla c(x)^T p\right]_+  - [c(x)]_+ +\phi_2 \beta^2 ||p||^2 \\
&\leq -\beta [c(x)]_+ +\phi_2 \beta^2 ||p||^2.
\end{align*}

Let us now write
\begin{align*}
F_r(x + \beta p) - F_r(x) &= \Psi(x + \beta p) - \Psi(x) + r\left([c(x + \beta p)]_+ - [c(x)]_+\right) \\
&\leq - \frac{\beta\lambda}{2} p^T J(x)^TJ(x)p + \lambda \beta c(x) + \phi_1 \beta^2 ||p||^2 - r\beta [c(x)]_+ + r \phi_2 \beta^2 ||p||^2\\
&\leq - \frac{\beta\lambda}{2} p^T J(x)^TJ(x)p + \beta(\lambda- r) [c(x)]_+  +   (\phi_1 + r \phi_2)\beta^2 ||p||^2 \\
&\leq - \frac{\beta\lambda}{2} p^T J(x)^TJ(x)p +  (\phi_1 + r \phi_2)\beta^2 ||p||^2. 
\end{align*}
The final inequality follows from the fact that $r>\lambda$. It is clear that \cref{eq:descent} holds for all values of $\beta$ that satisfy
\begin{equation*}
- \frac{\beta\lambda}{2} p^T J(x)^TJ(x)p +  (\phi_1 + r \phi_2)\beta^2 ||p||^2 \leq - \frac{\eta\lambda\beta}{2}  p^T  J(x)^T J(x) p.
\end{equation*}
By rearranging the terms it follows that the above equality holds for all $\beta$ with
\begin{equation*}
\beta \leq  \frac{(1-\eta)\lambda p^T J(x)^T J(x)p}{2(\phi_1 + r \phi_2)||p||^2}. 
\end{equation*}
This concludes the proof since the right-hand side in this equality is strictly positive. 
\end{proof}

\begin{theorem}\label{thm:conv}
Suppose we generate an infinite sequence of iterates $x_{k} = x_{k-1} + \beta_k p_k$ that all satisfy descent condition \cref{eq:descent} for penalty parameter $r_k = \max(2\lambda_k,r_{k-1})$, with $(p_k,\lambda_k)$ the KKT pair \cref{eq:kkt2} corresponding with $x_{k-1}$ and $r_0$ an initial penalty parameter. Suppose $J(x_{k-1})^T J(x_{k-1}) \succeq \mu I$ for some constant $\mu>0$ independent of $k$. Suppose in addition that the iterates $(x_k,\lambda_k)$ remain bounded and that there exists a constant $\gamma$ independent of $k$ such that $\beta_k \lambda_k > \gamma >0$. Then there exists a sub-sequence $\mathcal{K}\subset\mathbb{N}$ such that $(x_k,\lambda_k)_{k\in\mathcal{K}}$  converges to a solution of \cref{eq:kkt1}. 
\end{theorem}
\begin{proof}
By boundedness of $(x_k,\lambda_k)$ we know, using the Bolzano-Weierstrass theorem, that there exist a convergent sub-sequence $(x_k,\lambda_k)_{k\in\mathcal{K}}\rightarrow (\bar{x},\bar{\lambda})$ for some $\mathcal{K} \subset \mathbb{N}$. Hence, we can assume that $r_k = r$ for some constant $r$ for all $k\in\mathcal{K}$ sufficiently large. Since the descent condition \cref{eq:descent} holds and $\lambda_k\beta_k>\gamma$ we have for $k$ sufficiently large
\begin{equation*}
F_r(x_{k}) - F_r(x_{k-1}) \leq - \frac{\eta \lambda_k\beta_k}{2}  p_k^T  J(x_{k-1})^T J(x_{k-1}) p_k \leq - \frac{\eta\mu\gamma}{2}||p_k||^2.
\end{equation*}
In addition we have that the sequence $F_r(x_k)$ is decreasing and bounded below and thus it must converge. Hence, it also follows that $F_r(x_{k-1}) - F_r(x_{k})\rightarrow 0$. Combining this with the fact that $F_r(x_{k-1}) - F_r(x_{k}) \geq \frac{\eta\mu\gamma}{2}||p_k||^2$, we can conclude that $p_k\rightarrow 0$.  Writing out the KKT conditions \cref{eq:kkt2} for $(x_k,\lambda_k)$ we get
\begin{equation*}
\left\{ \begin{matrix}\nabla \Psi(x_k + p_k) + \lambda_k (\nabla c(x_k) + J(x_k)^T J(x_k) p_k) = 0, \\
c(x_k) + \nabla c(x_k)^T p_k + \frac{1}{2}p_k^T J(x_k)^T J(x_k)p_k = 0. \end{matrix} \right.
\end{equation*}
Now by taking the limit for $k\in\mathcal{K}$ and using the fact that $p_k\rightarrow 0$ we have
\begin{equation*}
\left\{ \begin{matrix}\nabla \Psi(\bar{x}) + \bar{\lambda}\nabla c(\bar{x}) = 0, \\
c(\bar{x})= 0, \end{matrix} \right.
\end{equation*}
which concludes the proof. 
\end{proof}

\begin{algorithm}
\caption{``exact'' Sequential Projected Newton method (exact SPN)}
\label{alg:SQCO} 
\begin{algorithmic}[1] \small
\STATE{Initialize $k=0, r_0 > 0$ and $\eta,\theta\in(0,1)$ and choose tolerance $\tau$.}
\STATE{Compute initial point $x_0$ with $||s(x_0) - b||\leq \sigma$ and set $\lambda_0 = -\frac{\nabla \Psi(x_0)^T \nabla c(x_0)}{||\nabla c(x_0)||^2}$. \label{lam0}}
\WHILE{$||F(x_{k},\lambda_{k})||\geq \tau$} 
\STATE{Solve \eqref{eq:kkt2} for $x=x_k$ and denote the solution as $(p_{k+1},\tilde{\lambda}_{k+1})$. \label{solve}}
\STATE{Update penalty parameter $r_{k+1} = \max(2\tilde{\lambda}_{k+1},r_{k})$.}
\STATE{Choose $\beta_{k+1}$ the largest value in $\left\{1,\theta, \theta^2, \theta^3,\ldots \right\}$ such that 
\begin{equation*}F_{r_{k+1}}(x_{k} +\beta_{k+1} p_{k+1}) - F_{r_{k+1}}(x_{k}) \leq  - \frac{\eta \tilde{\lambda}_{k+1}\beta_{k+1}}{2}||J(x_k)p_{k+1}||^2  \end{equation*} \label{backtrack}}
\STATE{Update $x_{k+1} = x_{k} + \beta_{k+1} p_{k + 1}$ and $\lambda_{k + 1} = \lambda_{k} + \beta_{k + 1}\left(\tilde{\lambda}_{k+1} - \lambda_{k}\right)$.}
\STATE{$k = k+1$}
\ENDWHILE
\end{algorithmic}
\end{algorithm}

\begin{remark}
Boundedness of the Lagrange multiplier $\lambda_k$ is the most important assumption in \cref{thm:conv}. Otherwise, the descent property \cref{eq:descent} is meaningless, since this would imply $r_k\rightarrow \infty$. Boundedness of $x_k$ is not necessary to prove that $p_{k}\rightarrow 0$, which can be seen by carefully reading the proof. By \cref{eq:littleo} we know that this implies that the sub-problem \cref{eq:reg2} becomes a better and better approximation of \cref{eq:reg1}. Hence, it might be possible to remove the boundedness assumption of $x_k$. However, since for practical purposes (which will become clear in the rest of the manuscript) we need to consider a more heuristic approach, we believe there is little added value in explicitly proving this.
\end{remark}

\Cref{thm:conv} allows us to formulate a line-search method for solving \eqref{eq:kkt1}, see \cref{alg:SQCO}. Convergence of the algorithm can be monitored using the nonlinear function
\begin{equation}\label{eq:Fdef}
F(x,\lambda) = \begin{pmatrix} \nabla \Psi(x) + \lambda \nabla c(x) \\ c(x) \end{pmatrix}
\end{equation}
that expresses the KKT conditions \cref{eq:kkt1} of the equality constrained optimization problem \cref{eq:reg1} (not to be confused with the merit function $F_r(x)$ used in the line-search). In \cref{sec:num} we also investigate other possible convergence metrics. 

Note that the definition of $\lambda_0$ on line \ref{lam0} can be seen as the least squares approximation of the Lagrange multiplier corresponding with $x_0$, see for instance \cite{boggs1989strategy}. First of all by \cref{thm:descent} we have that the backtracking line-search on line \ref{backtrack} is well-defined and terminates with some $\beta_{k+1}>0$. Furthermore we know by \cref{thm:conv} that this algorithm converges towards the solution of \cref{eq:kkt1} under certain assumptions. 

To improve the practical performance of \cref{alg:SQCO} we describe a slightly more heuristic approach. In principle we can use any constrained optimization algorithm to solve \eqref{eq:reg2} on line \ref{solve}. However, the Projected Newton (PN) method \cite{Cornelis_2020,cornelis2020projected} is particularly well suited to solve this problem, since it exploits some of the problem-specific properties. The Projected Newton method is itself an iterative method that uses Generalized Krylov subspaces to solve a sequence of projected problems. The number of floating point operations used in the Projected Newton method grows with the iteration index of the algorithm, so for performance it is important that the number of (inner) iterations required to converge remains relatively small. More specifically we have that the number of floating point operations in the Projected Newton method grows with $\mathcal{O}(k_{inner}^2)$, where $k_{inner}$ is the inner iteration index. Hence, if we require that \cref{eq:reg2} is solved very accurately, then \cref{alg:SQCO} will not be very efficient. In the numerical experiments section we illustrate that \cref{eq:reg2} does not have to be solved very accurately to be able to make progress towards the solution. Furthermore it turns out that \cref{eq:reg2} is a very accurate model of \cref{eq:reg1} in the sense that the descent property \cref{eq:descent} is satisfied by $\beta=1$ for all iterations in the experiments of \cref{sec:num}. Hence, to improve performance of the algorithm and remove the seemingly unnecessary line-search we describe a slightly more heuristic approach.

To formulate the computationally efficient variant of \cref{alg:SQCO} we define
\begin{equation} \label{eq:fk}
F^{(k)}(p,\lambda) = \begin{pmatrix} \nabla\Psi(x_k + p) + \lambda\left(\nabla c(x_k) + J(x_k)^T J(x_k) p\right) \\
c(x_k) + \nabla c(x_{k})^T p + \frac{1}{2} p^T J(x_k)^T J(x_k) p  \end{pmatrix}
\end{equation}
the system of nonlinear equations that express the KKT conditions \cref{eq:kkt2} of the sub-problem on line \ref{solve} of \cref{alg:SQCO} for $x=x_k$.
\Cref{alg:SPN} is a heuristic approach that works very well in practice. We illustrate this claim by a number of numerical experiments in \cref{sec:num}.
\begin{algorithm}
\caption{Sequential Projected Newton (SPN) method}
\label{alg:SPN} 
\begin{algorithmic}[1] \small
\STATE{Initialize $k=0$ and $\zeta > 1$ and choose tolerance $\tau$.}
\STATE{Compute initial point $x_0$ with $||s(x_0) - b||\leq \sigma$ and set $\lambda_0 = -\frac{\nabla \Psi(x_0)^T \nabla c(x_0)}{||\nabla c(x_0)||^2}$.}
\WHILE{$||F(x_{k},\lambda_{k})||\geq \tau$} 
\STATE{Apply Projected Newton method \cite{cornelis2020projected} to \begin{equation}\label{eq:subproblem}\min_{p\in\mathbb{R}^n} \Psi(x_k + p) \hspace{0.5cm}\text{subject to}\hspace{0.5cm}\frac{1}{2}||J(x_k)p  + s(x_k) - b||^2 = \frac{\sigma^2}{2}
 \end{equation}
with stopping criterion $||F^{(k)}(p_{k+1},\lambda_{k+1})||\leq \max\left\{\frac{||F(x_{k},\lambda_{k})||}{\zeta}, \frac{\tau}{10}\right\}$}
\STATE{Update $x_{k+1} = x_{k} + p_{k + 1}$}
\STATE{$k = k+1$}
\ENDWHILE
\end{algorithmic}
\end{algorithm}

\begin{remark} \label{thm:forcing_sequence}
The introduction of the parameter $\zeta$ in \cref{alg:SPN} is inspired by the class of inexact Newton methods \cite{dembo1982inexact}, which is a generalization of Newton's method for solving a nonlinear system of equations $F(v):\mathbb{R}^{n} \rightarrow \mathbb{R}^{n}$. In stead of solving the Jacobian linear system $F'(v_k) \Delta v_{k+1} = -F(v_k)$ exactly for some iterate $v_k\in \mathbb{R}^n$, an approximate solution is computed that satisfies $\|F'(v_k)\Delta v_{k+1} + F(v_k) \| \leq \eta_k \|F(v_k)\|$, with ``forcing term'' $\eta_k \in (0,1]$. Choosing a suitable forcing sequence can significantly improve practical performance \cite{eisenstat1996choosing,gomes2008globally}. 
\end{remark}

\section{Reference methods and related work} \label{sec:reference_methods}

As a first reference method we consider an algorithm that is closely related to the Sequential Quadratically Constrained Quadratic Programming (SQCQ) method developed in \cite{fukushima2003sequential}. It can be seen as a modified version of \cref{alg:SQCO}, where we replace sub-problem \eqref{eq:reg2} on line \ref{solve} with sub-problem \eqref{eq:quadmodel} that uses a quadratic model of $\Psi(x + p)$. Moreover, we also allow an approximate solution of the sub-problem using the $\zeta$ parameter, similarly as in \cref{alg:SPN}. Accuracy of the sub-problem is monitored using  
\begin{equation*} 
\tilde{F}^{(k)}(p,\lambda) = \begin{pmatrix} \nabla\Psi(x_k) + \nabla^2 \Psi(x_k) p + \lambda\left(\nabla c(x_k) + J(x_k)^T J(x_k) p\right) \\
c(x_k) + \nabla c(x_{k})^T p + \frac{1}{2} p^T J(x_k)^T J(x_k) p  \end{pmatrix}.
\end{equation*}
The resulting algorithm is shown in \cref{alg:spnq} and is referred to as SPN-Q (where Q stands for ``quadratic"). As previously mentioned, we can readily modify the analysis in \cref{sec:deriv} to show similar results for this modified version. In fact, the authors in \cite{fukushima2003sequential} provide a local and global convergence analysis of the closely related SQCP method. 

\begin{algorithm}
\caption{SPN-Q method}
\label{alg:spnq} 
\begin{algorithmic}[1] \small
\STATE{Initialize $k=0, r_0 > 0$ and $\eta,\theta\in(0,1)$ and choose tolerance $\tau$.}
\STATE{Compute initial point $x_0$ with $||s(x_0) - b||\leq \sigma$ and set $\lambda_0 = -\frac{\nabla \Psi(x_0)^T \nabla c(x_0)}{||\nabla c(x_0)||^2}$.}
\WHILE{$||F(x_{k},\lambda_{k})||\geq \tau$} 
\STATE{Apply Projected Newton method \cite{cornelis2020projected} to \begin{equation}\label{eq:subproblem2}\min_{p\in\mathbb{R}^n} p^T \Psi(x_k) + \frac{1}{2} \nabla^2 p^T \Psi(x_k) p\hspace{0.5cm}\text{subject to}\hspace{0.5cm}\frac{1}{2}||J(x_k)p  + s(x_k) - b||^2 = \frac{\sigma^2}{2}
 \end{equation}
with stopping criterion $||\tilde{F}^{(k)}(p_{k+1},\tilde{\lambda}_{k+1})||\leq \max\left\{\frac{||F(x_{k},\lambda_{k})||}{\zeta}, \frac{\tau}{10}\right\}$}
\STATE{Update penalty parameter $r_{k+1} = \max(2\tilde{\lambda}_{k+1},r_{k})$.}
\STATE{Choose $\beta_{k+1}$ the largest value in $\left\{1,\theta, \theta^2, \theta^3,\ldots \right\}$ such that 
\begin{equation*}F_{r_{k+1}}(x_{k} +\beta_{k+1} p_{k+1}) - F_{r_{k+1}}(x_{k}) \leq  - \frac{\eta \tilde{\lambda}_{k+1}\beta_{k+1}}{2}||J(x_k)p_{k+1}||^2  \end{equation*}}
\STATE{Update $x_{k+1} = x_{k} + \beta_{k+1} p_{k + 1}$ and $\lambda_{k + 1} = \lambda_{k} + \beta_{k + 1}\left(\tilde{\lambda}_{k+1} - \lambda_{k}\right)$.}
\STATE{$k = k+1$}
\ENDWHILE
\end{algorithmic}
\end{algorithm}

To describe the second reference method, we consider the unconstrained reformulation of the regularized nonlinear inverse problem with fixed regularization parameter given by \cref{eq:reg_alpha}.
The main difficulty with solving \cref{eq:reg_alpha} is the fact that the fit-to-data term $\| s(x) - b\|$ is in general a non-convex function. This means that the solution is not necessarily unique, but more importantly, this also means that the Hessian $\nabla^2 f(x)$ of the objective function $f(x)=\frac{1}{2} \| s(x) - b\|^2 + \alpha\Psi(x)$ can possibly be indefinite. This in its turn implies that the standard Newton direction $p=-\left(\nabla^2 f(x)\right)^{-1} \nabla f(x)$ is not necessarily a descent direction. However, this difficulty can be circumvented by in stead computing the Gauss-Newton direction
\begin{equation} \label{eq:solve_gn}
(J(x)^TJ(x) + \alpha \nabla^2 \Psi(x)) p  = -J(x)^T(s(x) - b) - \alpha\nabla\Psi(x). 
\end{equation}
A similar approach is for instance taken in \cite{haber2000optimization}, where the authors consider nonlinear inverse problems where the forward model is described by a system of partial differential equations. The right-hand side in \cref{eq:solve_gn} is precisely $-\nabla f(x)$ (i.e. the same as in Newton's method) and $J(x)^T J(x)$ is a positive-definite approximation of the true Hessian of the nonlinear least squares term $\frac{1}{2} \|s(x) - b\|^2$. It is then easy to show that $p\in\mathbb{R}^n$ is a descent direction for $f(x)$, i.e. that $p^T\nabla f(x)\leq 0$. 

This means that we can compute a suitable step-length such that we have a decrease in the objective function $f(x + \beta p) < f(x)$. The ideas above form the basis of the Gauss-Newton method, see \cref{alg:GN}. The linear system of equations \cref{eq:solve_gn} can be solved (approximately) used the Conjugate Gradient method \cite{hestenes1952methods}, which is one of the most well-known and efficient Krylov subspace methods \cite{liesen2013krylov,saad2003iterative} . The stopping criterion for the linear system of equations is chosen to be $1/\zeta$, in correspondence to the choice of tolerance used for the Projected Newton method in \cref{alg:SPN}, see also \cref{thm:forcing_sequence}. 

Unfortunately, due to the fact that the fit-to-data term \cref{eq:reg_alpha} is non-convex we can not use specialized algorithms such as the split Bregman method \cite{goldstein2009split} or ADMM \cite{wahlberg2012admm}. However, SPN-Q and the Gauss-Newton method can be used regardless of which forward model $s(x)$ is used. Hence, we will use \cref{alg:GN,alg:spnq} to compare the Sequential Project Newton method with in our numerical experiments.

For linear inverse problems there also exist a wide range of different algorithms that can be used, see for instance \cite{lampe2012large,doi:10.1137/140967982,GAZZOLA2014180,4380459,rodriguez2008efficient,doi:10.1137/18M1194456,doi:10.1137/130917673} and references therein. Hence, many of these algorithm could also be used to solve sub-problem \eqref{eq:reg2} in the Sequential Projected Newton method. However, efficiency of the Projected Newton method was already illustrated by a number of numerical experiments in \cite{cornelis2020projected}, so we do not consider alternative algorithms for the inner solver in our experiments. 
\begin{algorithm}
\caption{Gauss-Newton (GN) method}
\label{alg:GN} 
\begin{algorithmic}[1] \small
\STATE{Initialize $x_0 = 0, k=0$ and $\zeta > 1, \theta\in(0,1)$ and choose tolerance $\tau$.}
\WHILE{$||\nabla c(x_k) + \alpha \nabla \Psi(x_k)||\geq \tau$} 
\STATE{Apply the Conjugate Gradient method \cite{hestenes1952methods} to 
\begin{equation}\label{eq:solvegn} \left(J(x_k)^TJ(x_k) + \alpha \nabla^2 \Psi(x_k)  \right) p = -\left(\nabla c(x_k) + \alpha \nabla\Psi(x_k) \right)
 \end{equation}
with stopping criterion $\frac{\|(J(x_k)^TJ(x_k) + \alpha \nabla^2 \Psi(x_k)) p_{k + 1} + \left(\nabla c(x_k) + \alpha \nabla\Psi(x_k) \right)\|}{\|\nabla c(x_k) + \alpha \nabla\Psi(x_k) \|}  \leq  \frac{1}{\zeta}$ \vspace{0.3cm}}. 
\STATE{Choose $\beta_{k+1}$ the largest value in $\left\{1,\theta, \theta^2, \theta^3,\ldots \right\}$ such that 
\begin{equation*} \frac{1}{2}\|s(x_k + \beta_{k + 1}p_{k + 1}) - b\|^2 + \alpha \Psi(x_k + \beta_kp_{k + 1}) < \frac{1}{2}\|s(x_k) - b\|^2 + \alpha \Psi(x_k). \end{equation*} \vspace{-0.5cm}}
\STATE{Update $x_{k+1} = x_{k} + \beta_{k+1} p_{k + 1}$.}
\STATE{$k = k+1$}
\ENDWHILE
\end{algorithmic}
\end{algorithm}

\begin{remark}\label{thm:initialpoint} We can use the Gauss-Newton method (\cref{alg:GN}) with $\alpha = 0$ to compute the initial point $x_0$ used in algorithms \ref{alg:SQCO}, \ref{alg:SPN} and \ref{alg:spnq}. In this case, the Gauss-Newton method converges towards the unregularized solution $\argmin \|s(x) - b\|$. We terminate the algorithm when $\|s(x) - b\| < \sigma$, which typically happens already after only a few iterations. Hence, computing this initial point is relatively cheap. We use a fixed tolerance of $10^{-10}$ (i.e. $\zeta = 10^{10}$) for \cref{eq:solvegn} in all our experiments (to make sure all algorithms start from the same initial point, regardless of other parameter choices). 
\end{remark}

\section{Test problem: Talbot-Lau X-ray phase contrast imaging} \label{sec:talbot}

For our numerical experiments we consider a test-problem from Talbot-Lau phase constrast imaging \cite{von2017grating}. The forward model of the Talbot-Lau phase contrast projection operation in the discretized setting can be described with:
\begin{equation}\label{eq:forward}
s(\mu,\epsilon,\delta) =  s_0 e^{-A\mu}\left[1 + v_0 e^{-A\epsilon} \cos(\phi_0 + A_\phi \delta) \right]
\end{equation}
where $s_0\in\mathbb{R}^{+}_0$ is the incoming intensity, $v_0\in(0,1)$ the visibility and $\phi_0 \in\mathbb{R}^m$ the phase shift without the object. The variables $\mu,\epsilon$ and $\delta$ are the absorption, dark field and phase contrast, respectively. In the current manuscript we consider two-dimensional examples where the three components $\mu,\epsilon$ and $\delta$ can be visualized as pixel images. Similarly we could also consider three dimensional examples where the different components contain a large number of voxels. 
If we denote $\tilde{n}$ the total number of pixels (or voxels) then we have $\mu,\epsilon,\delta \in\mathbb{R}^{\tilde{n}}.$ 
Now let $m$ denote the total number of projections, i.e. the number of projection angles times the number of detectors. Then we have $s(\mu,\epsilon,\delta)\in\mathbb{R}^m$. In addition, we denote $b$ the vector of length $m$ containing the acquired (noisy) projections.

The matrix $A\in\mathbb{R}^{m\times \tilde{n}}$ is the conventional projection matrix and $A_{\phi}\in\mathbb{R}^{m\times \tilde{n}}$ the system matrix for differential phase contrast: $A_{\phi}=DA$, where $D \in\mathbb{R}^{m \times m}$ is a matrix performing the finite difference operation. We generate the matrix $A$ using the ASTRA toolbox \cite{van2015astra,van2016fast} based on a simple parallel beam geometry with line kernel, see \cref{sec:matlab}. Let us denote $x = \left[\mu^T, \epsilon^T,\delta^T \right]^T \in \mathbb{R}^{n}$ with $n=3\tilde{n}$ the vector obtained by stacking the three components.  Note that $s_{i}(x)$ is a function $\mathbb{R}^{n} \longrightarrow \mathbb{R}$ and $s(x)$ is a function $\mathbb{R}^{n} \longrightarrow \mathbb{R}^m$.

To use the techniques developed in this manuscript we first need to compute the Jacobian matrix $J(x)$ of the nonlinear forward model $s(x)$. 
To do so it is convenient to write $s(x) = \tilde{s}(\tilde{x})$, with $\tilde{x} = \bar{A}x = \left[t^T,u^T,v^T \right]^T \in\mathbb{R}^{3m}$,
\begin{equation*}
\bar{A}=\begin{pmatrix} A & 0 & 0 \\ 0 & A & 0 \\ 0 & 0 & A_\phi  \end{pmatrix} \text{ and } \tilde{s}_i(\tilde{x}) = s_0 e^{-t_i}\left[1 + v_0 e^{- u_i} \cos(\phi_0 + v_i) \right],\hspace{0.2cm} (i=1,\ldots,m).
\end{equation*}
Let us calculate the Jacobian $J_{\tilde{s}} (\tilde{x})$ of the function $\tilde{s}(\tilde{x})$. 
For $i \neq j$ we have that 
$\frac{\partial \tilde{s}_i}{\partial t_{j}} = \frac{\partial \tilde{s}_i}{\partial u_{j}} = \frac{\partial \tilde{s}_i}{\partial v_{j}} = 0$
and for $i = j$ we have
\begin{equation*}
\frac{\partial \tilde{s}_i}{\partial t_{i}} = -\tilde{s}_i, \hspace{0.2cm} 
\frac{\partial \tilde{s}_i}{\partial u_{i}} = -s_0 v_0 e^{- t_i} e^{- u_i} \cos(\phi_0 + v_i) \hspace{0.2cm} \text{and} \hspace{0.2cm}
\frac{\partial \tilde{s}_i}{\partial v_{i}} = -s_0 v_0 e^{- t_i} e^{- u_i} \sin(\phi_0 + v_i).
\end{equation*}
Hence we can write $J_{\tilde{s}} (\tilde{x}) = \begin{pmatrix} -D_1(\tilde{x}) & -D_2(\tilde{x}) & -D_3(\tilde{x}) \end{pmatrix} \in\mathbb{R}^{m \times 3m}$ with 
\begin{eqnarray*}
D_1(\tilde{x})  &=& \text{diag}\left(\tilde{s}_i\right)_{i = 1,\ldots,m} \\
D_2(\tilde{x})  &=& \text{diag}\left(s_0 v_0 e^{- t_i} e^{- u_i} \cos(\phi_0 + v_i) \right)_{i = 1,\ldots,m} \\
D_3(\tilde{x})  &=& \text{diag}\left(s_0 v_0 e^{- t_i} e^{- u_i} \sin(\phi_0 + v_i) \right)_{i = 1,\ldots,m}. 
\end{eqnarray*}
Using the chain rule for the Jacobian $s(x)  = \tilde{s}(\bar{A}x)$ we immediately get
\begin{equation*}
J(x) = J_{\tilde{s}} (\bar{A}x) \bar{A} = \begin{pmatrix} -D_1(\bar{A} x)A & -D_2(\bar{A} x)A & -D_3(\bar{A} x)A_\phi \end{pmatrix}.
\end{equation*}

\subsection{Total variation regularization}
Total variation regularization is a popular regularization technique that preserves edges in images. Let $y \in\mathbb{R}^{\tilde{n}}$ be a vector obtained by stacking all columns of a pixel image $Y\in\mathbb{R}^{N\times N}$ with $\tilde{n} = N^2$. The anisotropic total variation function is defined as
\begin{equation} \label{eq:tv}
\text{TV}(y) = \sum_{i,j = 1}^{N} |\partial^{(i,j)}_h(Y)| + |\partial^{(i,j)}_v(Y)|.
\end{equation}
with finite difference operators in the horizontal and vertical direction given by 
\begin{equation*} \partial^{(i,j)}_h(Y) = \left\{ \begin{matrix} Y_{i,j+1} - Y_{i,j}  & j<N \\ 0 & j=N \end{matrix} \right. \hspace{0.5cm}\text{and}\hspace{0.5cm} \partial^{(i,j)}_v(Y) = \left\{ \begin{matrix} Y_{i+1,j} - Y_{i,j}  & i<N \\ 0 & i=N \end{matrix} \right. . 
\end{equation*}
We can rewrite this in a more convenient way by first writing 
\begin{equation*}
\tilde{D} = \begin{pmatrix}
1 & - 1& & & \\ & \ddots&\ddots&&\\ && 1&-1
\end{pmatrix} \in\mathbb{R}^{(N-1) \times N}
\end{equation*}
which represents a finite difference approximation of the derivative operator in one dimension.
Let  $\otimes$ denote the Kronecker product. We can compactly write $\text{TV}(y) = ||\tilde{L}y||_1$ with 
\begin{equation*}
\tilde{L} = \begin{pmatrix}D_h \\D_v \end{pmatrix} \in\mathbb{R}^{(2\tilde{n} -2N) \times \tilde{n} } \hspace{0.5cm}\text{and} \hspace{0.5cm}  \left\{ \begin{matrix} D_h =  \tilde{D}  \otimes I_N \in \mathbb{R}^{(\tilde{n} - N) \times \tilde{n} } \\
D_v = I_N \otimes \tilde{D} \in \mathbb{R}^{(\tilde{n}  - N) \times \tilde{n} } \end{matrix}\right. . 
\end{equation*}
The matrices $D_h$ and $D_v$ represent the two dimensional finite difference approximation of the derivative operator in the horizontal and the vertical directions respectively. Now because we have three different images $\mu,\epsilon$ and $\delta$ we consider the regularization term $||Lx||_1$ with 
\begin{equation*}
L = \begin{pmatrix} \tilde{L} & 0 & 0 \\ 0 & \tilde{L} & 0 \\ 0 & 0 & \tilde{L} \end{pmatrix} \in\mathbb{R}^{\tilde{m}\times n}
\end{equation*}
and $\tilde{m} = 3(2\tilde{n} - 2N)$. The $\ell_1$ norm is of course non-differentiable, however, it is easy to
formulate a smooth approximation $\Psi(x) \approx ||Lx||_{1}$, where $\Psi: \mathbb{R}^n\rightarrow
\mathbb{R}$ is a twice continuously differentiable convex function. More precisely we define
\begin{equation}\label{eq:smoothlp}
\Psi(x) = \sum_{i=1}^{\tilde{m}} \sqrt{[Lx]_i^2 + \xi}
\end{equation}
where $\xi > 0$ is a small scalar that ensures smoothness. This is the same technique as used in \cite{cornelis2020projected}.

\section{Numerical experiments}\label{sec:num}

In this section we perform a number of numerical experiments using the Talbot-Lau test-problem described in \cref{sec:talbot}. The aim is to illustrate that \cref{alg:SPN} is a robust and computationally efficient method for solving the nonlinear system of equations $F(x,\lambda) = 0$ defined by the KKT equations \cref{eq:Fdef} for the equality constrained optimization problem \cref{eq:reg1}. This function is also used to monitor convergence of the algorithm. 
All experiments in this section are performed using MATLAB on a laptop computer with Intel(R) Core(TM) i7-7700HQ CPU @ 2.80GHz. 

For all our experiments we set $s_0 = 1$ in the forward model \cref{eq:forward}, which can be seen as a normalization of the projection data. Unless explicitly stated otherwise, the value $v_0 = 0.75$ is used for the visibility and $\phi_0$ is generated using three phase steps. We refer to \cref{sec:matlab} for more details on how the test-problems are generated. For a certain exact solution $x_{ex}=(\mu_{ex},\epsilon_{ex},\delta_{ex})$ we generate exact data $b_{ex} = s(x_{ex})$ and then add Gaussian distributed white noise to obtain noisy data $b$. We refer to the value $\sigma = ||b - b_{ex}||$ as the noise-level and to $\sigma/||b_{ex}||$ as the relative noise-level. We denote \texttt{npix} the number of pixels in each direction (we only consider square images) and \texttt{nangles} the number of projection angles used and we set the number of detectors equal to \texttt{npix}. For these choices we have $n = 3\times \texttt{npix}^2$ unknowns and $m = \texttt{npix} \times \texttt{nangles}$ measurements. We always choose \texttt{nangles} to be a multiple $(\geq 3)$ of \texttt{npix}, which implies $m\geq n$. In each iteration of the SPN algorithm we start the Projected Newton method with initial Lagrange multiplier $10^5$, which works well in practice. 

\textbf{Experiment 1.} We start by performing a small experiment with some exact images $\mu_{ex},\epsilon_{ex}$ and $\delta_{ex}$ of size $10\times 10$, i.e. $\texttt{npix} =10$, and $\texttt{nangles} = 4\times \texttt{npix}$, resulting in a projection matrix $A$ of size $400\times 100$ ($m = 400,\,n = 300$). Next we generate data $b$ with relative noise-level $0.01$ in MATLAB as follows

\begin{lstlisting}
npix = 10; nangles = 4*npix, nstep = 3; v0 = 0.75; nn = npix^2; 
x_ex = randn(3*nn,1)/10;
[A,D,phi0] = generate_testproblem(npix,nangles,nstep) %Appendix A. 
b_ex = exp(-A*x_ex(1:nn)).* (1+v0*exp(-A*x_ex(nn+1:2*nn)).*cos(phi0+D*(A*x_ex(2*nn+1:end))));
noiselevel = 0.01; noise = randn(m, 1); 
b = b_ex + noiselevel*norm(b_ex)*noise/norm(noise);
\end{lstlisting}

We apply the ``exact'' Sequential Projected Newton method, \cref{alg:SQCO}, with $\Psi(x) = \sum_{i=1}^{n} \sqrt{x_i^2 + \xi}$ and parameters $r_0 = 10^3,\,\eta = 10^{-4},\,\theta = 0.5,\,\tau=10^{-6}$ and $\xi= 10^{-6}$. Note that this choice of regularization term can be seen as a smooth approximation of the $\ell_1$-norm, which is not a good choice for the random exact solution that we generated above. However, with the current experiment we are not yet concerned with the quality of the obtained solution, but only with the convergence behavior.  We use the Projected Newton method \cite{cornelis2020projected} to solve sub-problem  and set the tolerance for this solver to $\tau/10$, i.e. we stop the algorithm when we have found a point that satisfies $||F^{(k)}(p_{k+1},\tilde{\lambda}_{k+1})||\leq \tau/10$, where $F^{(k)}(p,\lambda)$ is defined by \cref{eq:fk}. For comparison we also apply SPN-Q (\cref{alg:spnq}) using the same parameters and with $\zeta = \infty$.

\begin{figure}
\begin{center}
\includegraphics[width=1\textwidth]{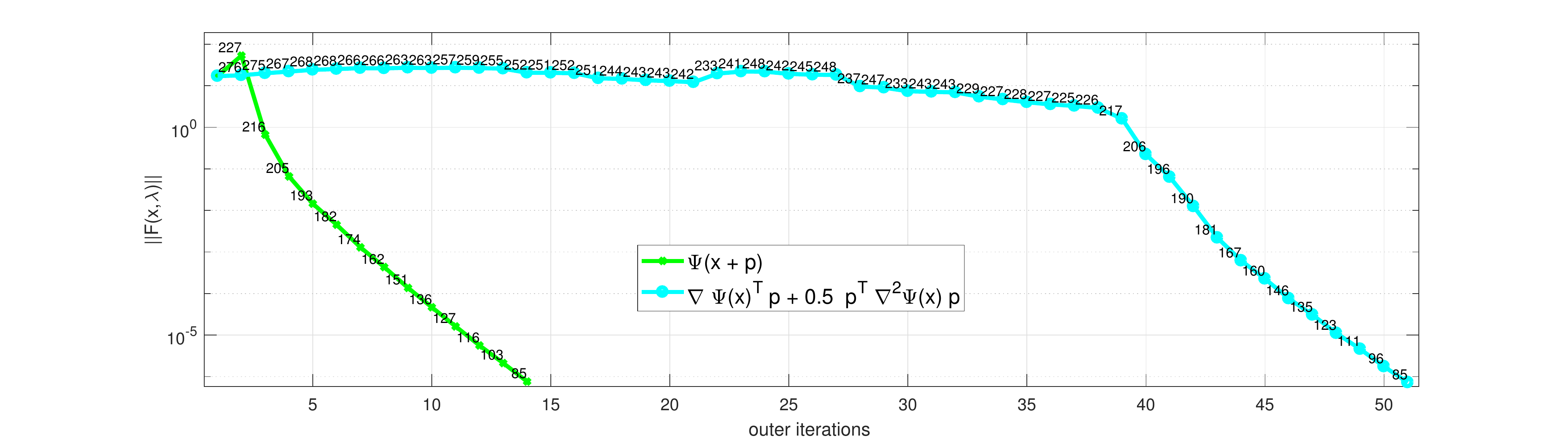}
\caption{\textbf{Experiment 1.} Convergence history of the exact SPN method (\cref{alg:SQCO}) compared with SPN-Q (\cref{alg:spnq}) with $\zeta = \infty$. Full step-sizes $\beta_k$ satisfy the decrease condition when we solve sub-problem \cref{eq:reg2} in \cref{alg:SQCO}, while smaller step-sizes are needed in case we use the quadratic model \cref{eq:quadmodel} in \cref{alg:spnq}. The numbers reported on the figure are the number of inner iterations required for that particular outer iteration. \label{fig:experiment1} \vspace{-1cm}}
\end{center}
\end{figure}

The result of this experiment is given by \cref{fig:experiment1}. We compare the value $||F(x_k,\lambda_k)||$ for both methods in terms of the outer SPN iterations $k$. In addition we also report the number of inner Projected Newton iterations required to converge for the sub-problems. First of all it is important to note that the decrease condition on line \ref{backtrack} of \cref{alg:SQCO} is satisfied by the step-length $\beta_k = 1$ for all iterations $k$ in case we solve sub-problem \eqref{eq:reg2}. In contrast, in case we use a quadratic model \eqref{eq:quadmodel}, the line-search terminates with $\beta_k<1$ for a large number of iterations. This is the reason that the convergence in the latter case is much slower. It is also interesting to observe that the quadratic model requires more Projected Newton iterations on average to satisfy the stopping criterion. Hence it is clear from this experiment why choosing sub-problem \eqref{eq:reg2} over \eqref{eq:quadmodel} is the better choice. 

\textbf{Experiment 2.} In the next experiment we illustrate the fact that the sub-problem  \eqref{eq:reg2} does not have to be solved to a high accuracy to be able to make progress towards the solution, i.e. to decrease $||F(x,\lambda)||$. We keep the same set-up as before and consider two different relative noise-levels, namely $0.1$ and $0.01$. We apply \cref{alg:SPN} with $\zeta = \infty$ and $\tau = 10^{-4},10^{-6},10^{-8}$ and $10^{-10}$.  This can be seen as applying \cref{alg:SQCO}, i.e. taking a fixed tolerance for the sub-problems, but always talking a full step-length (which satisfy the decrease condition on line \ref{backtrack} anyway). To clearly observe the effect of the tolerance of the Projected Newton method on convergence, we let the algorithm run for a fixed number of iterations. The result of this experiment is given by \cref{fig:experiment2}. The final achievable accuracy in terms of $||F(x,\lambda)||$ is completely determined by the tolerance of the inner-solver. Furthermore, it can be observed that convergence (up to stagnation of the method) does not differ much in case we use a lower tolerance for the sub-problems.  This implies that in early iterations it is not necessary to compute a very accurate solution. Hence an adaptive tolerance as implemented using the $\zeta$ co\"efficient in \cref{alg:SPN} seems like a good idea to improve performance of the algorithm. This is especially important since the computational cost of the Projected Newton method grows with the number of iterations. 
With the next experiment we investigate the trade-off between keeping the number of inner iterations as low as possible without compromising the convergence in terms of the number of outer iterations needed to converge. 
\begin{figure}
\begin{center}
\begin{tabular}{cc}
\includegraphics[width=0.5\textwidth]{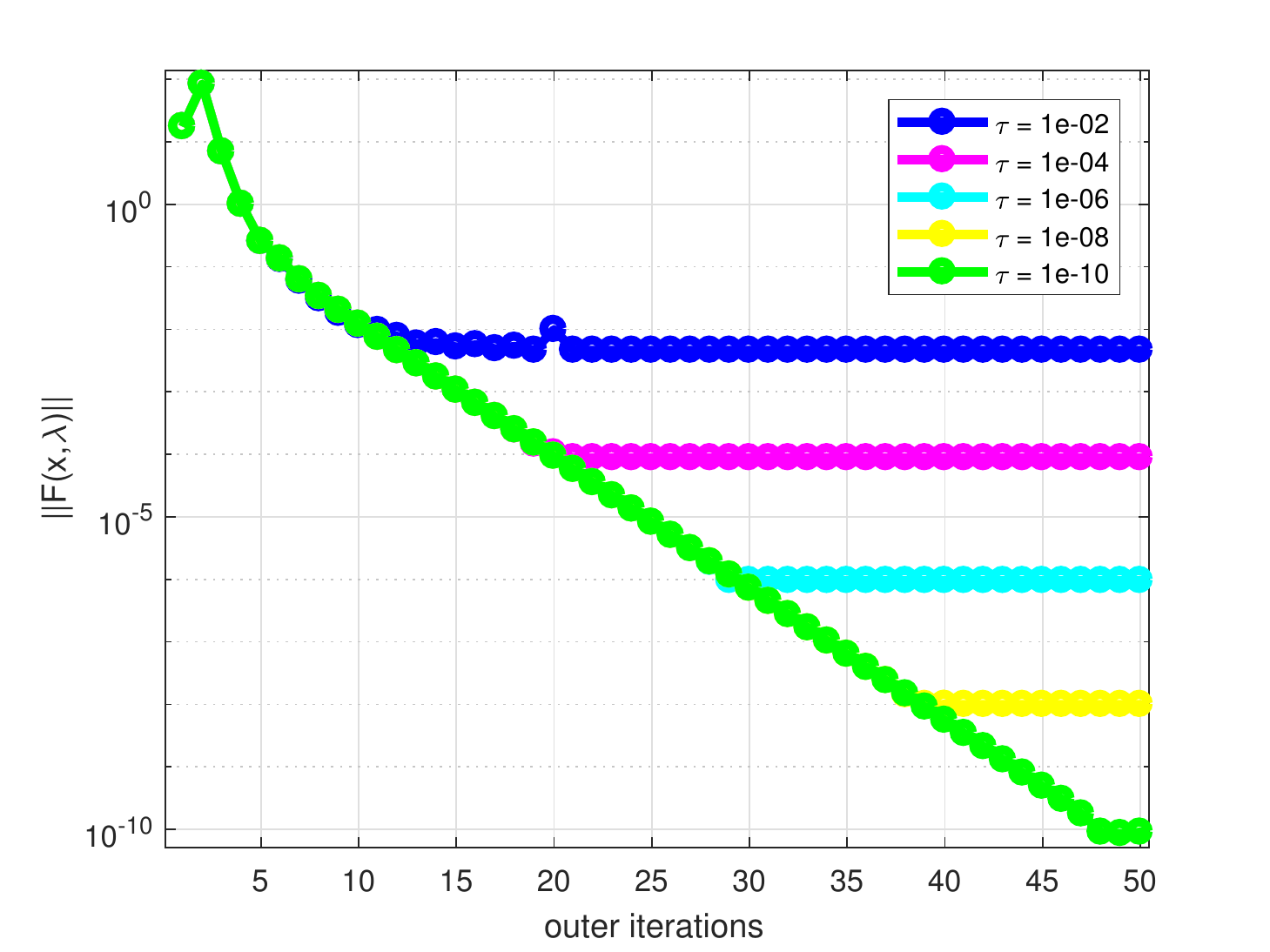} & \includegraphics[width=0.5\textwidth]{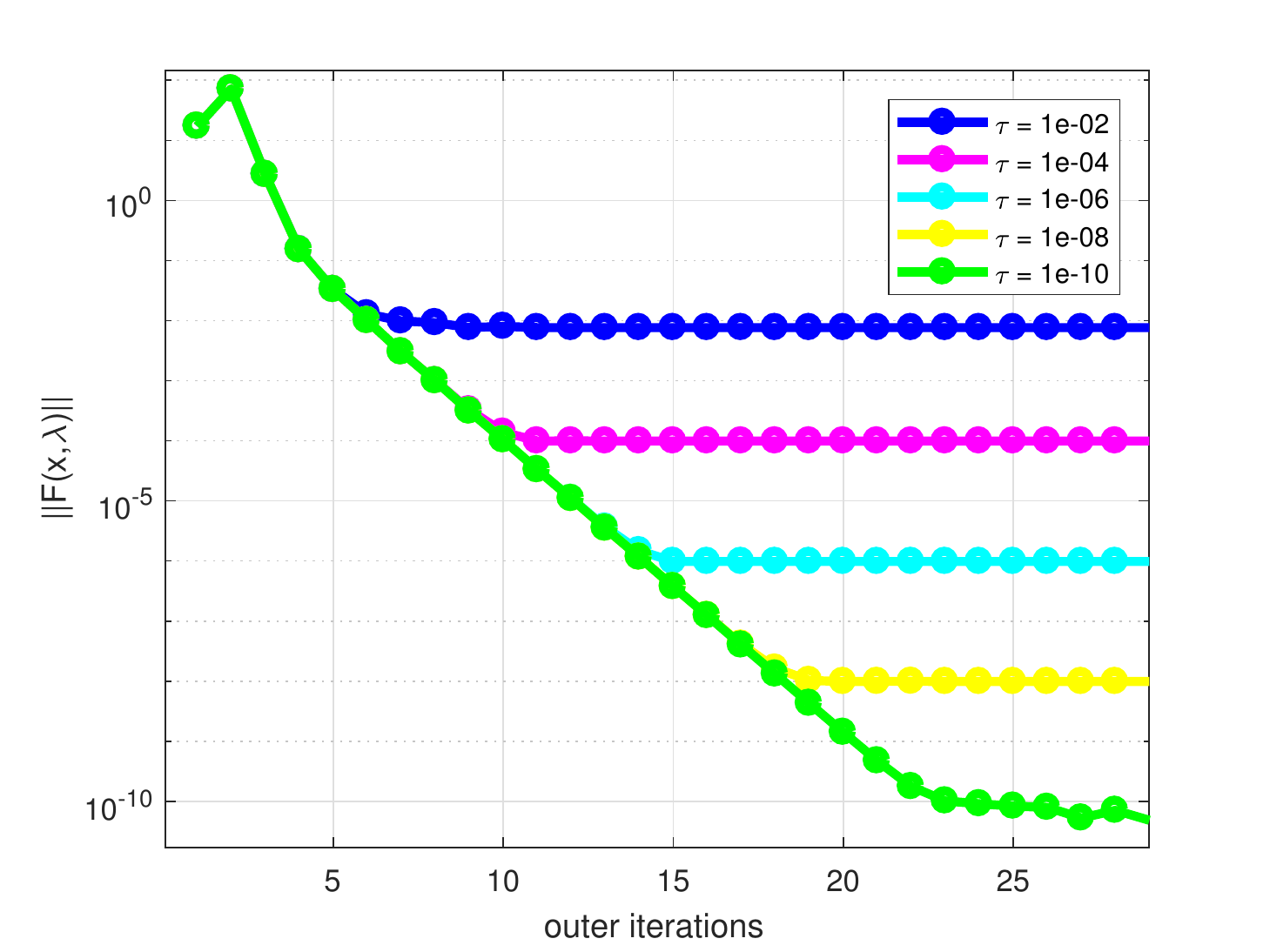} 
\end{tabular}
\caption{ \textbf{Experiment 2}. Influence of the tolerance of the Projected Newton inner-solver on the convergence history in terms of outer iterations of the SPN method (\cref{alg:SPN}) for two test-problems with relative noise-level 0.1 (left) and 0.01 (right). \label{fig:experiment2}}
\end{center}
\end{figure}

\textbf{Experiment 3.} In this experiment we study the effect of the parameter $\zeta>1$ in \cref{alg:SPN}. If this value is large, then we require a relatively accurate solution of the sub-problem and we can expect that we need a large number of inner iterations to converge. However, we can also expect to make significant progress towards the solution in terms of $||F(x_k,\lambda_k)||$ for the outer iterations. To confirm this we apply the Sequential Projected Newton method, \cref{alg:SPN}, with different values of the $\zeta$ parameter to a Talbot-Lau test-problem with  $\texttt{npix} = 16$ and $\texttt{nangles} = 5 \times \texttt{npix}$, i.e. $m = 1280,\,n = 768$. Furthermore we take relative noise-level $0.01$, visibility $v_0 = 0.5$, tolerance $\tau = 10^{-6}$ and smoothing parameter $\xi = 10^{-6}$. 

\begin{figure}
\begin{center}
\includegraphics[width=1\textwidth]{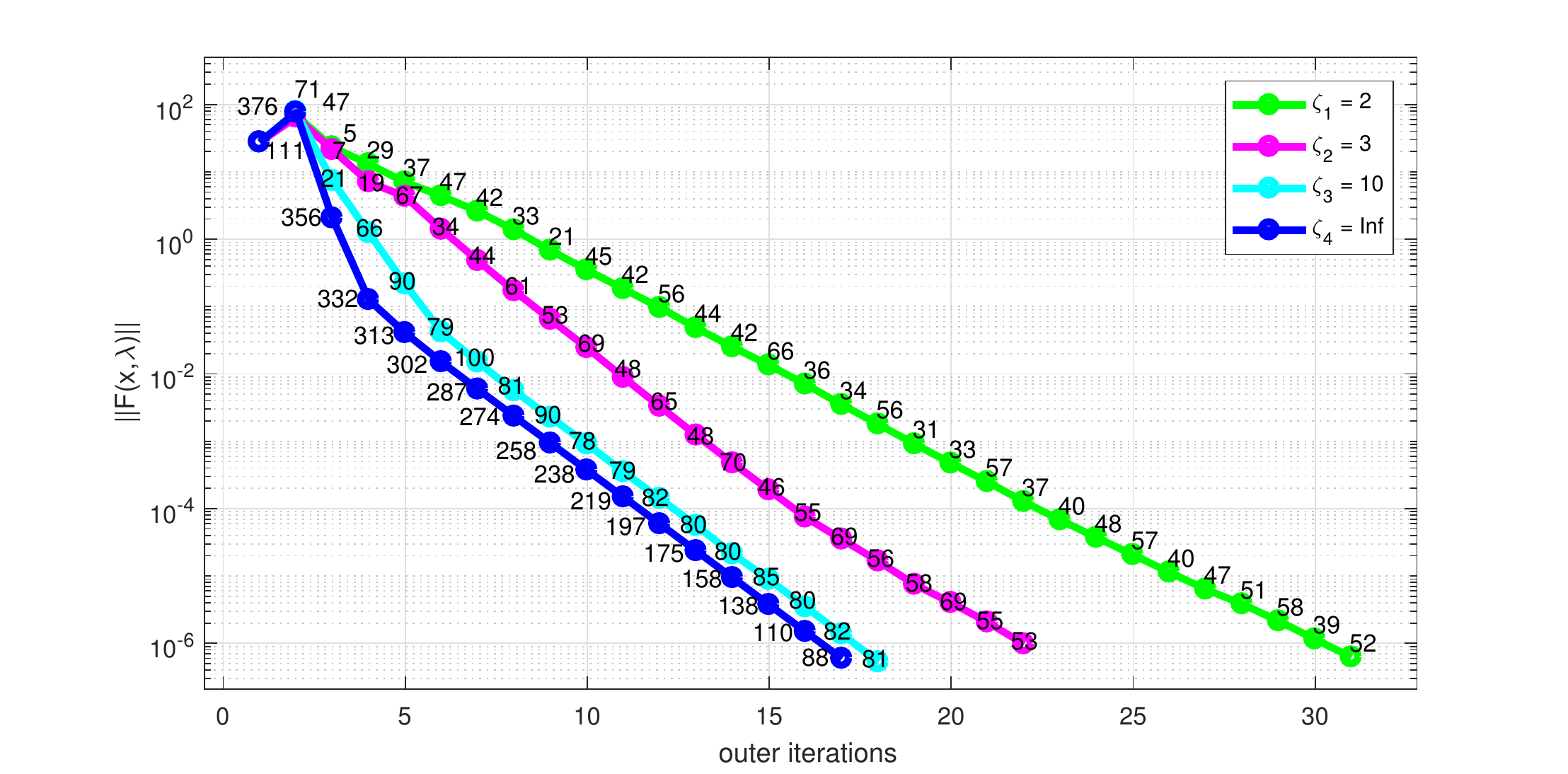}
\caption{\textbf{Experiment 3.}  Influence of the $\zeta$ parameter on the convergence history of the SPN method (\cref{alg:SPN}) in terms of the outer iterations. The numbers reported on the figure are the number of inner iterations required for that outer iteration.\label{fig:experiment3}}
\end{center}
\end{figure}

In \cref{fig:experiment3}  we show convergence of the method for $\zeta = 2,3,10$ and $\infty$. When we take a small value such as $\zeta = 2$  we can observe that the number of inner iterations needed to satisfy the stopping criterion is relatively small. However, convergence in terms of the outer iterations is also quite slow. For the choice $\zeta = \infty$ we get very rapid convergence in terms of outer iterations since we require an accurate solution ($\tau = 10^{-6}$) of the sub-problem, but we need a lot inner iterations to satisfy the stopping criterion of the inner-solver. It is also interesting to observe that in this case the number of Projected Newton iteration decreases towards convergence. This is because the Projected Newton method starts from an initial point containing all zeroes and near the solution of $F(x,\lambda) = 0$ we have that this is indeed a good initial ``guess'' for $p_k$.

To investigate this trade-of in a bit more detail we apply \cref{alg:SPN} again with a linearly spaced range of $100$ values for $\zeta$ from $1.1$ to $30$. The number of outer iterations required to converge are reported in \cref{fig:zeta_its} (left). If we choose $\zeta$ close to one, then we need to perform a lot of outer iterations. However, from a certain point onward, the number of outer iterations hardly changes, which implies that there is almost no benefit to solving the sub-problem very accurately.  

In \cref{fig:zeta_its} (right) we show the total number of inner (Projected Newton) iterations. Again, when $\zeta$ is close to one, we need to perform a lot of inner iterations, which is due to the fact that the number of outer iterations is quite large. A value of $\zeta$ is a little less than 5 seems to result in the fewest number of inner iterations for this particular experiment. From that point onwards, the number of inner iterations starts to increase again, since we require a more accurate solution of the sub-problem, without requiring fewer outer iterations to converge.
 
\begin{figure}
\begin{center}
\begin{tabular}{cc}
\includegraphics[width=0.5\textwidth]{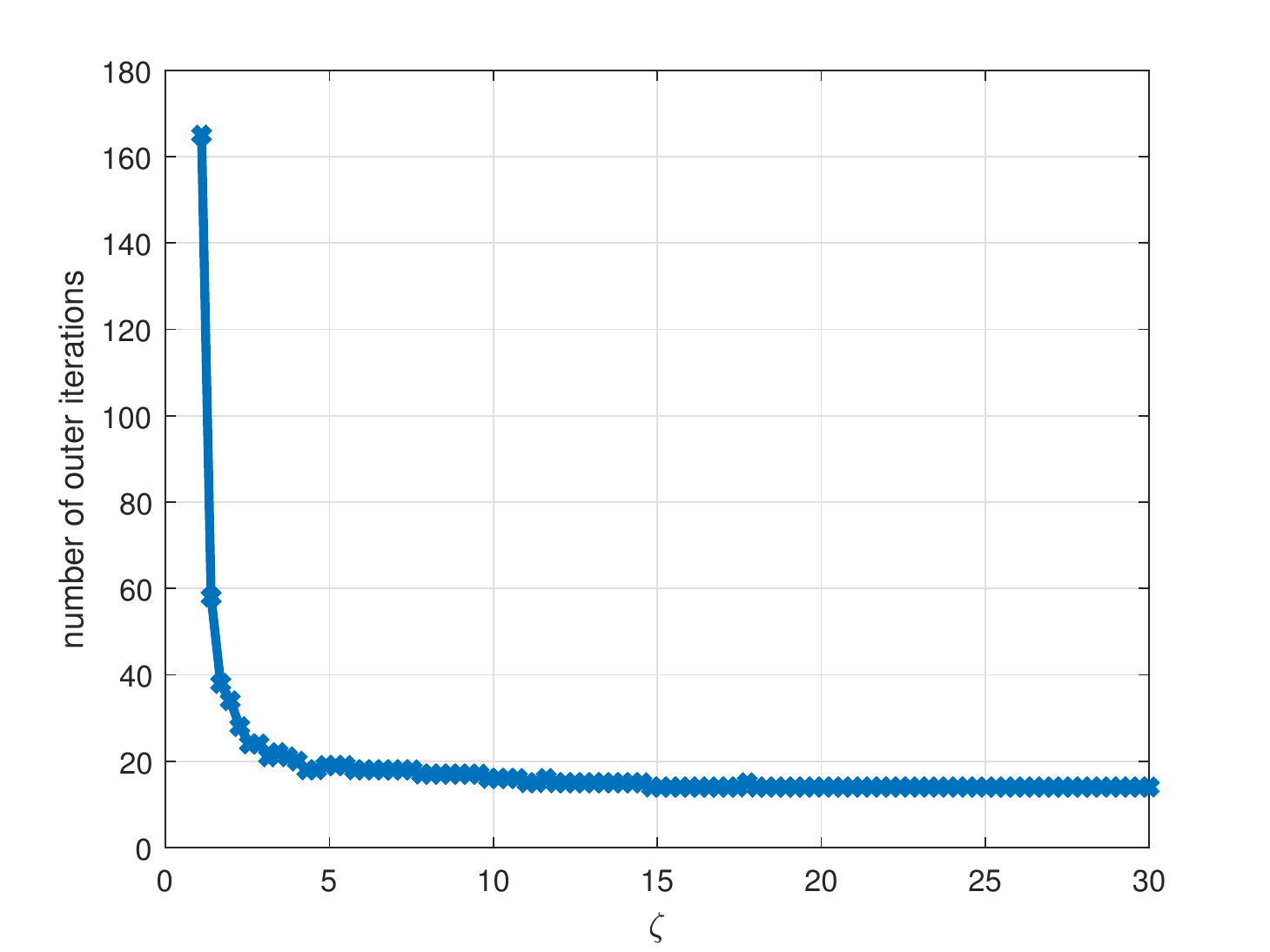} & \includegraphics[width=0.5\textwidth]{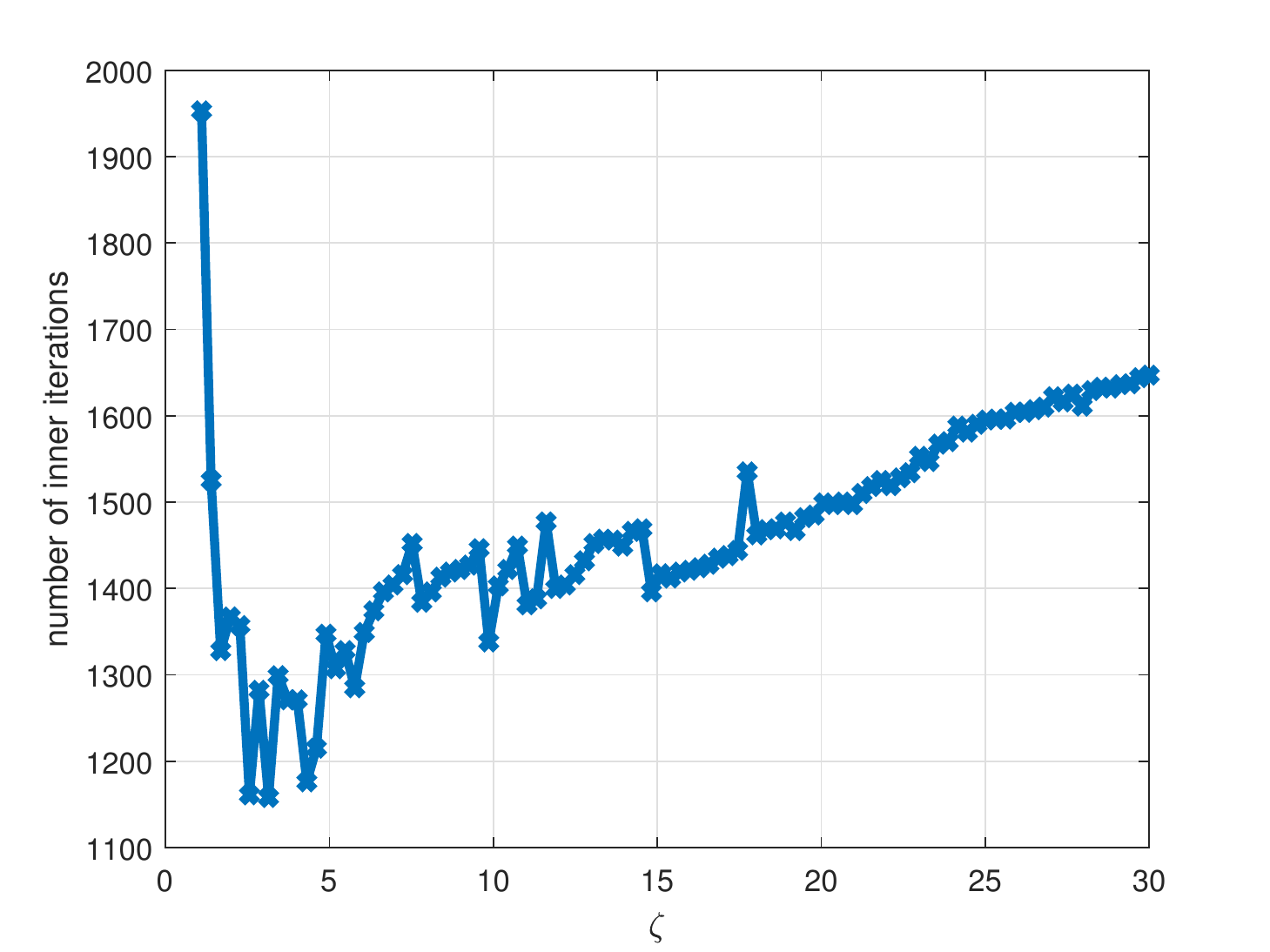} 
\end{tabular}
\caption{ \textbf{Experiment 3}. Influence of the $\zeta$ parameter in the SPN method (\cref{alg:SPN}) on the number of outer iterations required for convergence (left) and total number of inner iterations required to solve \cref{eq:subproblem} using the Projected Newton method (right). \label{fig:zeta_its}}
\end{center}
\end{figure}

\begin{table}
\centering
\scalebox{0.95}{
\begin{tabular}{ l|r | rrrrrrrrr}
						        &$\zeta$         & 1.1   & 1.5   & 2     & 3     & 4     & 5     & 10    & $\infty$ \\  \hline 
                    & outer its      & 113   & 41    & 26    & 16    & 13    & 12    & 8     & 5     \\ 
\texttt{npix} = 16  & avg. inner     & 4.67  & 10.63 & 16.85 & 23.69 & 28.08 & 32.50 & 44.75 & 169.80 \\ 
                    & runtime        & 0.59  & 0.50  & 0.50  & 0.45  & 0.45  & 0.48  & 0.47  & 3.29  \\ \hline 
                    & outer its      & 132   & 45    & 27    & 17    & 14    & 12    & 9     & 5\\
\texttt{npix} = 32  &avg. inner its  & 5.64  & 14.00 & 20.44 & 29.65 & 37.00 & 41.00 & 55.44 & 254.80 \\ 
                    &runtime         & 3.75  & 3.05  & 2.67  & 2.70  & 3.00  & 2.93  & 3.46  & 29.43 \\ \hline   
                    & outer its      & 157   & 47    & 28    & 18    & 15    & 13    & 9     & 6 \\
\texttt{npix} = 64  & avg. inner     & 8.38  & 19.68 & 29.71 & 41.44 & 52.00 & 59.92 & 75.44 & 337.33 \\ 
                    & runtime        & 40.40 & 28.06 & 26.21 & 25.10 & 26.67 & 28.47 & 26.98 & 241.25    
\end{tabular}}
\caption{ \textbf{Experiment 4}. The number of outer iterations (outer its), average number of inner iterations per outer iteration (avg. inner) and the total runtime (in seconds) of the SPN method (\cref{alg:SPN}) applied to a Talbot-Lau test-problem with \texttt{npix} = 16, 32 and 64. Reported timings are averaged over 10 runs with the same instance of noise. Th number of outer and inner iterations do not change over the different runs. \label{table:timing_zeta} }
\end{table}

\textbf{Experiment 4.}
Note that since the cost of the Projected Newton method grows (quadratically) with the inner iteration index, it does not necessarily mean that more inner iterations implies a longer total runtime. Fewer inner iterations on average, but more outer iterations might be preferable to few outer iterations that require a lot of inner iterations each. Hence in this experiment we take a closer look at the average number of inner iterations and the overall runtime of the method. To study the performance of the Sequential Projected Newton method (\cref{alg:SPN}) for different choices of $\zeta$ we report some actual timings in \cref{table:timing_zeta}. We use the same set-up as in experiment 3 but in addition to the choice $\texttt{npix} = 16$ $(m = 1280,\,n = 768)$, we now also take $\texttt{npix} = 32$ ($m = 5120,\,n=3072)$ and $\texttt{npix} = 64$ ($m = 20480,\,n=12228)$. The timings reported in \cref{table:timing_zeta} are averaged over 10 consecutive runs of the algorithm. We do not change the noise in the data for the different runs. Hence, the average number of inner iterations and the number of outer iterations do not change.  

It is clear that the choice $\zeta = \infty$ results in the fewest number of outer iterations. However, it also requires the most inner (Projected Newton) iterations on average. Recall that the number of floating point operations used in the Projected Newton method increases with the number of iterations, so it is important to keep this number as low as possible. Hence, this choice for $\zeta$ performs very badly in terms of total runtime. The timings for $\zeta$ between 1.1 and 10 are all comparable, except for the case $\texttt{npix} = 64$ and $\zeta=1.1$, which seems to perform a bit worse. The choice $\zeta = 10$ seems to keep the number of outer iterations close to that of the choice $\zeta=\infty$, while also keeping the average number of inner iterations modest. This value is used in all subsequent experiments. 

\begin{table}
\centering
\scalebox{1}{
\begin{tabular}{ l|r | rrrrr|r}
						 & run         & 1       & 2     & 3     & 4     & 5     & Std dev      \\  \hline 
SPN          & outer its   & 10      & 10    & 10    & 10    & 10    & 0   \\ 
             & runtime     & 27.17   & 30.01 & 29.75 & 29.06 & 30.03 & 1.20 \\ \hline 
SPN-Q        & outer its   & 87      & 91    & 152   & 89    & 93    & 27.82 \\ 
             & runtime     & 482.79  & 329.73& 721.85& 343.79& 354.73& 165.64   \\ \hline   
Gauss-Newton & outer its   & 88      & 74    & 63    & 88    & 65    & 12.05       \\ 
             & runtime     & 79.96   & 61.54 & 56.77 & 82.38 & 56.16 & 12.80
\end{tabular}}
\caption{ \textbf{Experiment 5}. The number of outer iterations (outer its) and total runtime (in seconds) of the SPN method (\cref{alg:SPN}), SPN-Q (\cref{alg:spnq}) and Gauss-Newton (\cref{alg:GN}) for 5 different runs (with different instances of noise) applied a the Tablot Lau test-problem with $\texttt{npix} = 64$. The standard deviation (Std dev) over the different runs is also reported for the different quantities. \label{table:batm}}
\end{table}

\begin{figure}
\begin{center}
\includegraphics[width=0.75\textwidth]{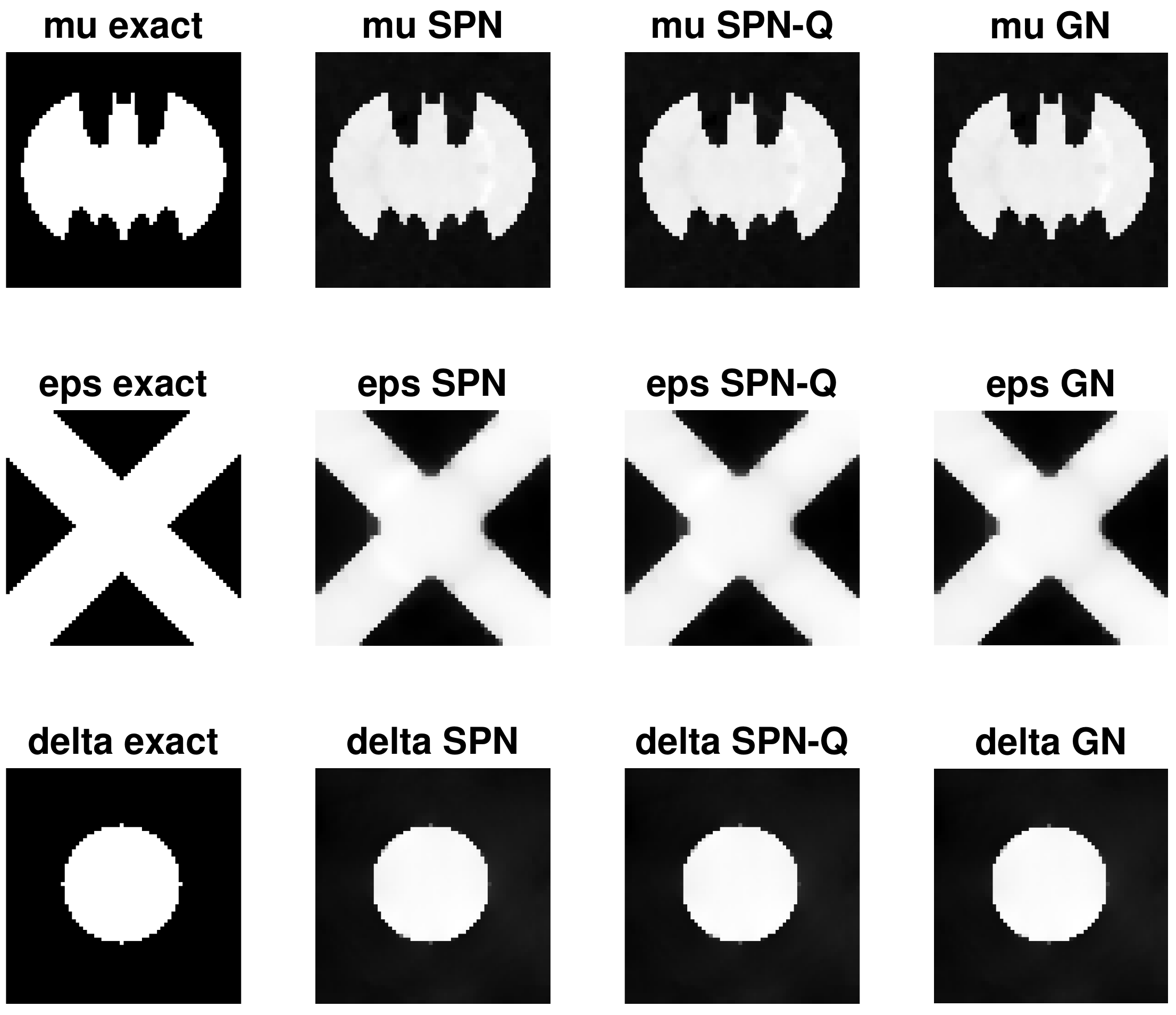}
\caption{\textbf{Experiment 5.} The exact solution used in the experiment together with the solutions obtained by \cref{alg:SPN} (SPN), \cref{alg:spnq} (SPN-Q) and \cref{alg:GN} (GN) with total variation regularization.\label{fig:experiment4sols}}
\end{center}
\end{figure}

\textbf{Experiment 5.} For the next experiment we consider total variation for the regularization function, as implemented in \cref{eq:smoothlp}. We use the exact images for $\mu_{ex},\epsilon_{ex}$ and $\delta_{ex}$ with $\texttt{npix} = 64$ as shown in \cref{fig:experiment4sols} and we take $\texttt{nangles} = 5 \times \texttt{npix}$ ($m = 20480,\,n=12228)$. For the other parameters we take tolerance $\tau = 10^{-6}$, relative noise-level $0.01$, $\zeta = 10$ and smoothing parameter $\xi = 10^{-8}$. Note that it is important that the smoothing parameter is relatively small, otherwise \cref{eq:smoothlp} is not a good approximation of the true total regularization function \cref{eq:tv}. 

To study the performance of the Sequential Projected Newton method (\cref{alg:SPN}) we also apply the two reference methods in \cref{sec:reference_methods}, namely SPN-Q (\cref{alg:spnq}) and Gauss-Newton (\cref{alg:GN}) to the same test-problem. All parameters (such as $\zeta$ or $\tau)$ for the different methods are kept the same (when applicable). The regularization parameter $\alpha$ for the Gauss-Newton method is chosen to be $1/\lambda^*$, where $\lambda^*$ is the final Lagrange multiplier computed using the Sequential Projected Newton method. In \cref{table:batm} we report the total runtime (in seconds) as well as the number of outer iterations needed to converge. We perform 5 different runs with 5 different instances of noise. It can be observed that \cref{alg:SPN} always performs best, both in number of iterations as well as runtime. Moreover it is clear that SPN-Q is not competitive compared to the other two approaches. Hence, the benefit of the more accurate sub-problem \cref{eq:subproblem} compared to the quadratic model \cref{eq:subproblem2} is clear from this experiment.

The reconstructed solutions are shown in \cref{fig:experiment4sols}.
Note that all methods (SPN, SPN-Q and GN) produce a solution of similar quality, which is not surprising since they are all essentially solving the same regularized nonlinear inverse problem using the total variation regularization function. 
We can clearly see that the reconstructed solutions are quite good and that the total variation regularization function nicely removes noise in the reconstruction, while still preserving edges in the images. 
\begin{figure}
\begin{center}
\begin{tabular}{cc}
\includegraphics[width=0.5\textwidth]{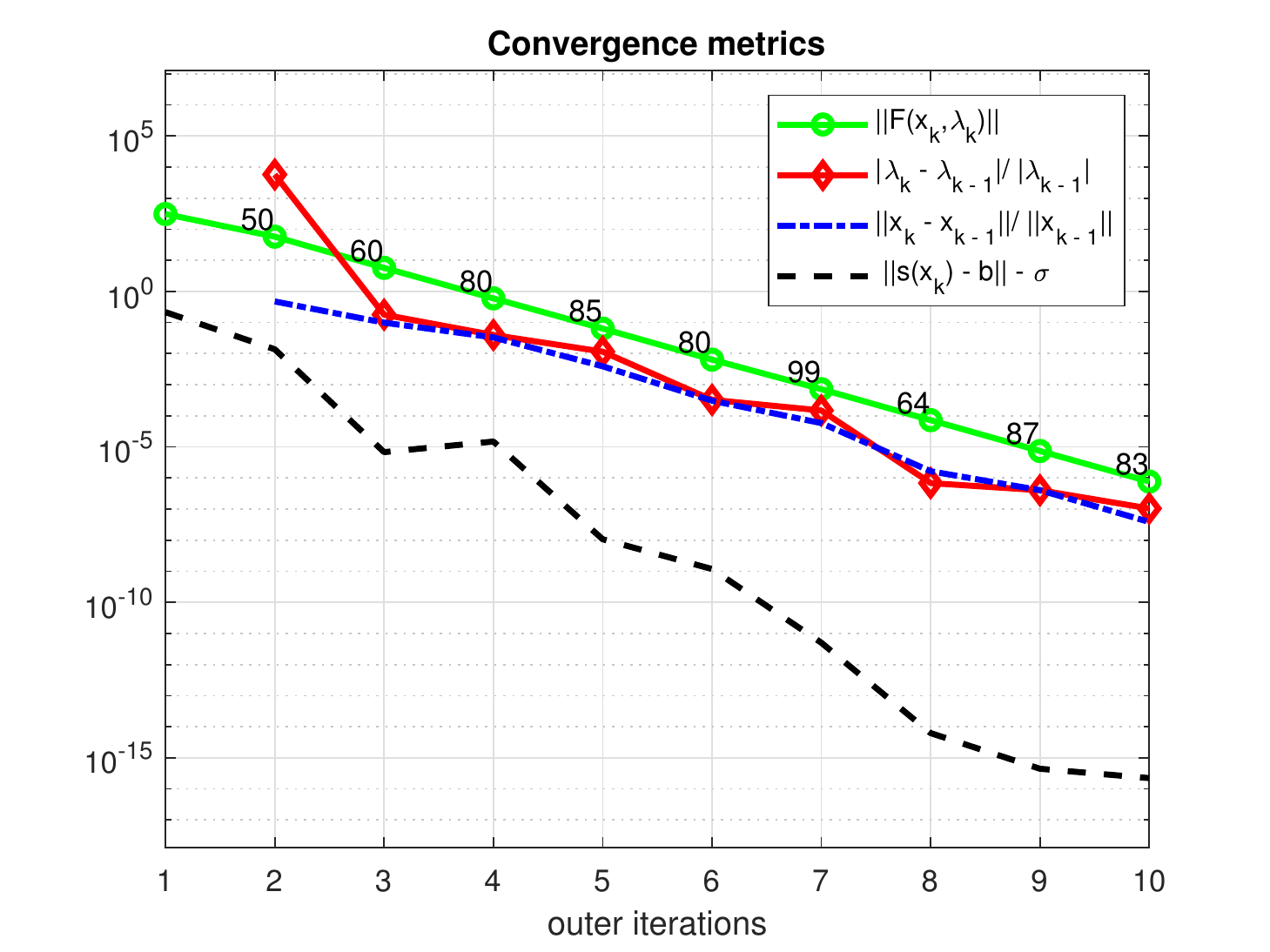} & \includegraphics[width=0.5\textwidth]{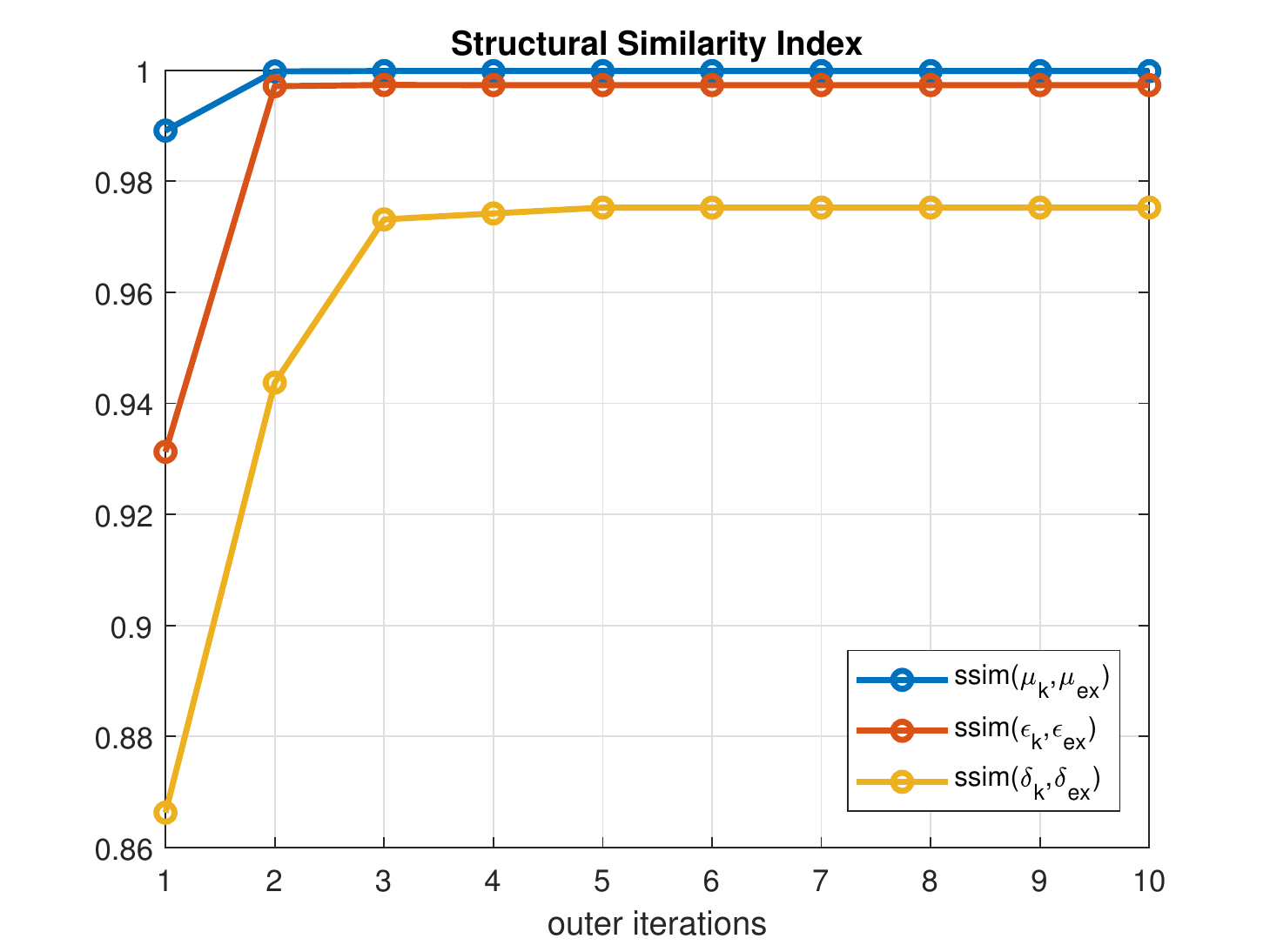} 
\end{tabular}
\caption{\textbf{Experiment 5.} (Left) Convergence metrics for the SPN method (\cref{alg:SPN}) of a Talbot-Lau experiment with $\texttt{npix} = 64$ and $\texttt{nangles} = 5 \times \texttt{npix}$. Tolerance is set to $\tau = 10^{-6}$ and noise-level 0.01. (Right) Structural Similarity index for the different components $\mu_k,\epsilon_k$ and $\delta_k$ for the same experiment. \label{fig:experiment4metrics}}
\end{center}
\end{figure}

\begin{figure}
\begin{center}
\includegraphics[width=0.75\textwidth]{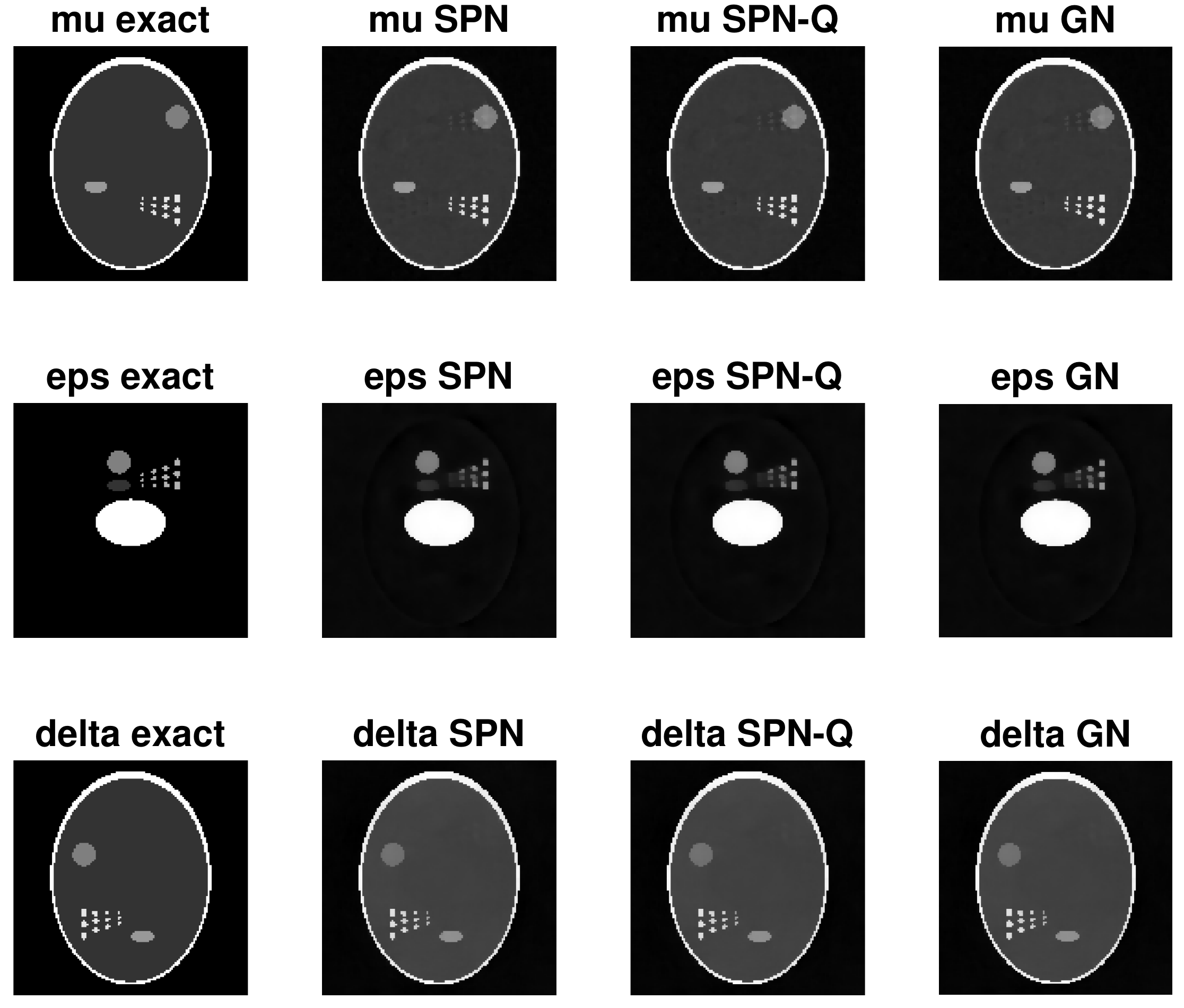}
\caption{\textbf{Experiment 6.}  The exact solution used in the experiment together with the solutions obtained by \cref{alg:SPN} (SPN), \cref{alg:spnq} (SPN-Q) and \cref{alg:GN} (GN) with total variation regularization. \label{fig:experiment5sols}}
\end{center}
\end{figure} 

\begin{figure}
\begin{center}
\begin{tabular}{cc}
\includegraphics[width=0.5\textwidth]{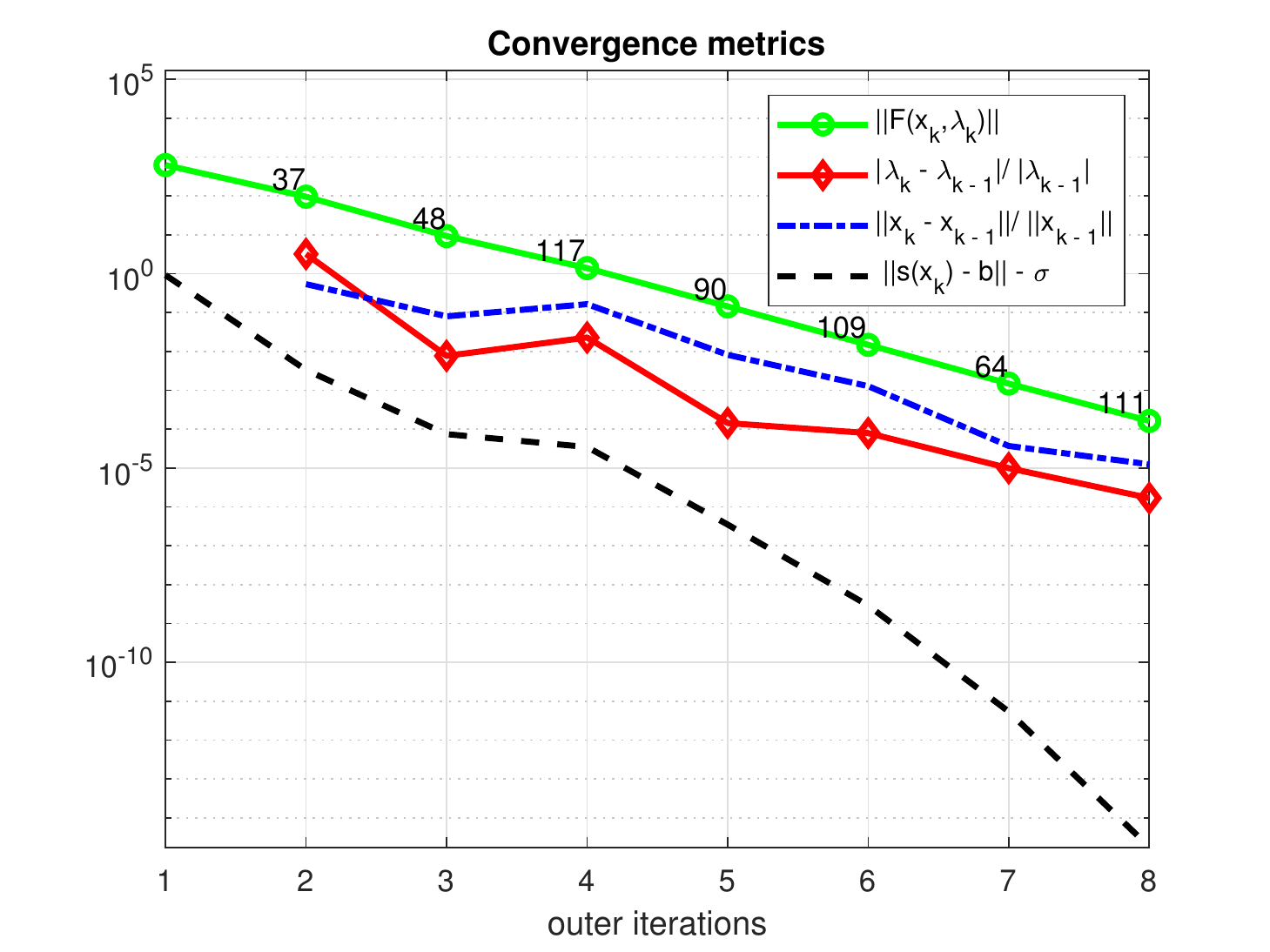} & \includegraphics[width=0.5\textwidth]{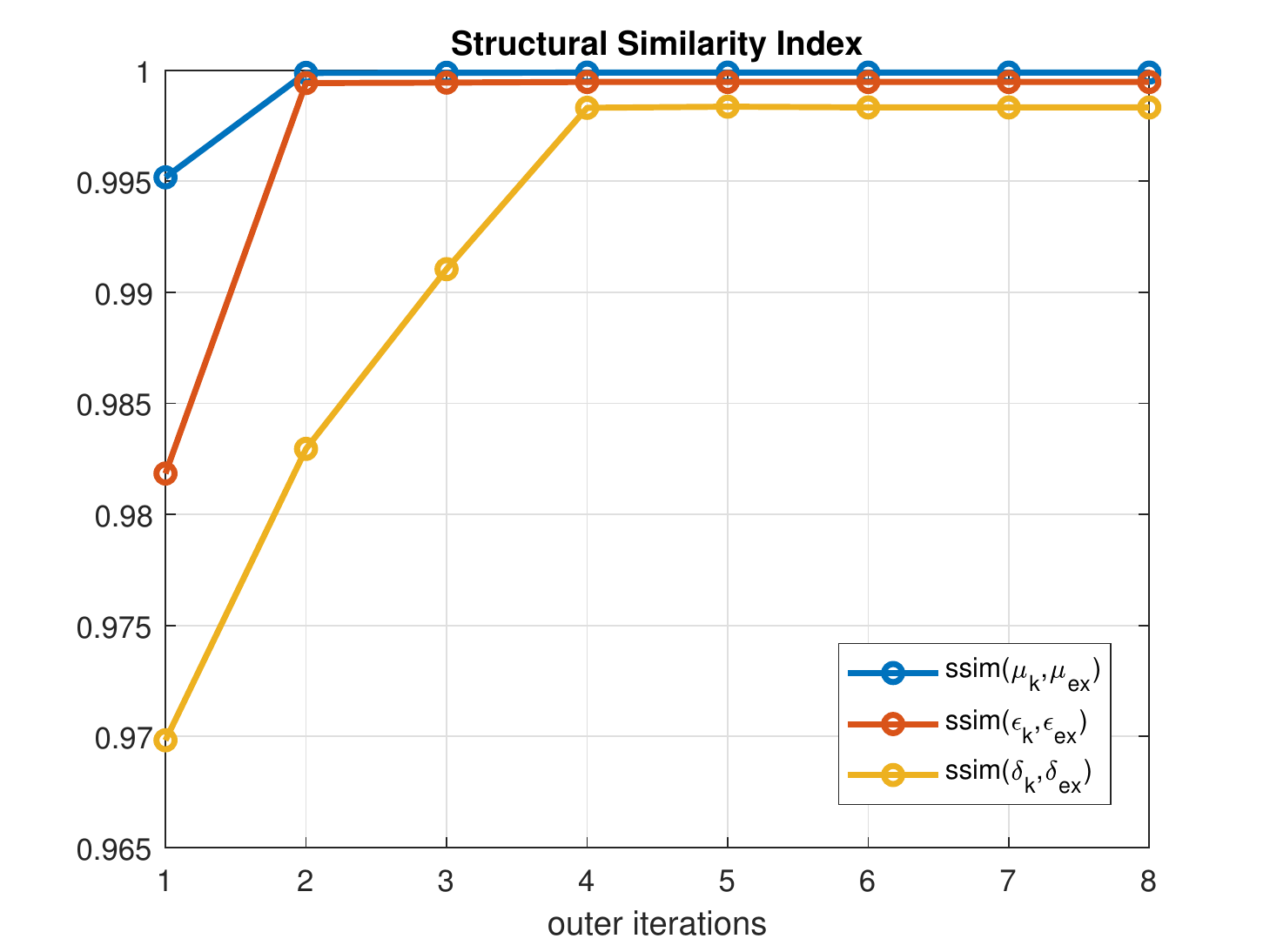} 
\end{tabular}
\caption{\textbf{Experiment 6.} (Left) Convergence metrics for the SPN method (\cref{alg:SPN}) of a Talbot-Lau experiment with $\texttt{npix} = 128$ and $\texttt{nangles} = 5 \times \texttt{npix}$. Tolerance is set to $\tau = 10^{-3}$ and noise-level 0.01. (Right) Structural Similarity index for the different components $\mu_k,\epsilon_k$ and $\delta_k$ for the same experiment. \label{fig:experiment5metrics}}
\end{center}
\end{figure}

In \cref{fig:experiment4metrics} (left) we report some different convergence metrics for the SPN method (\cref{alg:SPN}) applied to a particular instance of the test-problem. We show convergence of the values $||F(x_k,\lambda_k)||$ as always, as well as the relative difference of the iterates $x_k$ and $\lambda_k$. In addition we also show the convergence of the constraint by plotting $||s(x_k) - b || - \sigma$. These first three metrics seem to converge in a similar fashion, while the latter value converges much more quickly.

In the figure on the right, we show the Structural Similarity Index (SSIM) of the reconstructed solution compared with the exact solution. The closer this value is to $1$ the more similar the two images are. The actual formula for the SSIM is quite involved so we do not explicitely write it down here but rather refer to \cite{1284395} where it was first introduced. We compute this value using the MATLAB function \texttt{ssim}. It is interesting to see that the Structural Similarity Index stabilizes very quickly, even when $||F(x_k,\lambda_k)||$ is not yet very small. Hence, it might be appropriate for practical problem to use a stopping criterion based on $||s(x_k) - b || - \sigma$, which decreases much more quickly. Similar observations were made by the authors in \cite{cornelis2020projected}. 

\textbf{Experiment 6.} For our final experiment we use entirely the same set-up as before, but now use a larger test-problem with $\texttt{npix} = 128$ ($m = 81920,\,n = 49152$) as shown in \cref{fig:experiment5sols}. Furthermore, we now choose a moderate tolerance $\tau = 10^{-3}$. Results of this experiment can be found in \cref{fig:experiment5sols,fig:experiment5metrics}. The same conclusions as in the previous experiment can be made here. 
Results comparing \cref{alg:SPN} with Gauss-Newton are presented in \cref{table:table_phant}. We do not include the SPN-Q method in this table since a single run took more than one hour to complete. Again we can conclude that the Sequential Projected Newton method converges much more quickly than the Gauss-Newton method, while producing a solution of similar quality. 

\begin{table}
\centering
\scalebox{1}{
\begin{tabular}{ l|r | rrrrr|r}
						 & run         & 1        & 2      & 3      & 4     & 5     & Std dev      \\  \hline 
SPN          & outer its   & 8        & 8      & 8      & 8     & 8      & 0   \\ 
             & runtime     & 145.15   & 148.30 & 155.66 & 144.43& 166.80 & 9.36 \\ \hline  
Gauss-Newton & outer its   & 47       & 63     & 57     & 52    & 56     & 5.96      \\ 
             & runtime     & 451.26   & 556.58 & 570.99 & 637.92& 449.12 & 81.74
\end{tabular}}
\caption{ \textbf{Experiment 6}. The number of outer iterations and total runtime (in seconds) of the SPN method (\cref{alg:SPN}) and Gauss-Newton (\cref{alg:GN}) for 5 different runs (with different instances of noise) applied a the Tablot Lau test-problem with $\texttt{npix} = 128$. The standard deviation (Std dev) over the different runs is also reported for the different quantities.\label{table:table_phant}}
\end{table}

\section{Conclusions and outlook} \label{sec:conclusion}

In this work we have developed an algorithm for the automatic regularization of nonlinear inverse problems based on the discrepancy principle. By considering a linear approximation of the forward model using the Jacobian matrix, we obtain a quadratically constrained optimization problem which can be solved using the Projected Newton method. We prove that the solution of this equality constrained optimization problem results in a descent direction for a certain merit function, which can be used to formulate a formal line-search method for solving the inverse problem. We also formulate a slightly more heuristic version of the algorithm by loosening the tolerance of the solution of the sub-problems and removing the (seemingly unnecessary) line-search. We illustrate using the numerical experiments that these simplifications lead to a robust and computationally efficient method (in terms of runtime) for solving the regularized nonlinear inverse problem compared to two reference methods. Moreover, we show that we can obtain high quality reconstructions for Talbot-Lau X-ray phase contrast imaging using the proposed algorithm. 

Since the Sequential Projected Newton method uses the discrepancy principle, it can only be applied in case we have a good estimate of the noise-level available. It might be worth investigating if similar techniques as presented in this manuscript can also be used in combination with other parameter choice methods, such as generalized cross validation or the L-curve criterion. Moreover, we have assumed throughout this entire work that the nonlinear inverse problem is overdetermined. However, in many applications (including Tablot-Lau phase contrast imaging) this is not always the case. Hence, a generalization of the Sequential Projected Newton in case of an underdetermined inverse problem would certainly be valuable. Preliminary experiments (not reported in this work) indicate that the Sequential Projected Newton method also often converges nicely in the underdetermined case, but can also diverge on some test-problems. Hence, further research and safeguards to ensure convergence are required. 
\section*{Acknowledgments}

This work was funded in part by the IOF-SBO
project entitled ‘High performance iterative reconstruction methods for Talbot-Lau grating
interferometry based phase contrast tomography’.

\section*{References}
\bibliographystyle{unsrt}
\bibliography{refs_SPN}

\newpage
\appendix

 \section{MATLAB code for generating Talbot-Lau test-problem}\label{sec:matlab}

The following MATLAB code is based on the ASTRA toolbox \cite{van2015astra,van2016fast}. 
\vspace{-0.3cm}
\begin{lstlisting}
function [A,D,phi0] = generate_testproblem(npix,nangles,nstep)
% Matlab function to generate Talbot-Lau test-problem. 
% Input : - npix = number of pixels in each direction
%         - nangles = number of projection angles
%         - nstep = number of phase steps
% 
% Output: - m x n matrix A, with m = npix*nangles and n = npix^2
%         - m x m finite difference matrix D 
%         - m x 1 vector phi0: phase shift without the object
angles = linspace2(0,pi,nangles);
vol_geom = astra_create_vol_geom(npix,npix);
proj_geom = astra_create_proj_geom('parallel', 1.0, npix, angles);
proj_id = astra_create_projector('line', proj_geom, vol_geom);
matrix_id = astra_mex_projector('matrix', proj_id);
A = astra_mex_matrix('get', matrix_id); 
astra_mex_projector('delete', proj_id);
astra_mex_matrix('delete', matrix_id);

D1 = spdiags([-ones(npix,1), + ones(npix,1)],[0 1], npix,npix); 
I = speye(nangles);  D = kron(I,D1); 

phi_list = mod(0:(nangles - 1),nstep) + 1; %Index of the phase step
PHI_0 = linspace2(0,pi,nstep); PHIArray = zeros([npix nangles]);

for ii=1:npix, PHIArray(ii,:)= phi_list; end

PHIArray = PHIArray(:); phi0 = 0*PHIArray; 

for ii=1:length(phi0), phi0(ii) = PHI_0(PHIArray(ii)); end

end
\end{lstlisting}

\end{document}